\documentclass[a4paper,11pt]{article}
\usepackage[T1]{fontenc}
\usepackage{amsmath,amsfonts,amsthm,amssymb}
\usepackage{authblk}
\usepackage[hmargin={26mm,26mm},vmargin={30mm,35mm}]{geometry}
\usepackage[latin1]{inputenc}
\usepackage{newtxtext,newtxmath}
\usepackage{numprint} 
\usepackage{paralist}
\usepackage[colorlinks,allcolors={blue},linkcolor={black}]{hyperref}
\usepackage[bold]{hhtensor}
\usepackage{comment}
\usepackage{enumerate}
\usepackage{graphicx}
\usepackage{booktabs}
\usepackage[font=footnotesize]{caption}
\usepackage{tikz}

\usepackage[normalem]{ulem}
\newcounter{corr}
\definecolor{violet}{rgb}{0.580,0.,0.827}
\newcommand{\corr}[3]{\typeout{Warning : a correction remains in page
\thepage}
				\stepcounter{corr}        
				{\color{blue}\ifmmode\text{\,\sout{\ensuremath{#1}}\,}\else\sout{#1}\fi}
        {\color{red}#2}
        {\color{violet} #3}}

\newcommand{\email}[1]{\href{mailto:#1}{#1}}

\DeclareMathOperator{\optr}{tr}
\DeclareMathOperator{\opdev}{dev}

\newcommand{\GRAD}{\vec{\nabla}}
\newcommand{\GRADs}{\GRAD_\symm}

\newcommand{\DIV}{\vec{\nabla}\cdot}

\newcommand{\GRADh}{\GRAD_h}

\newcommand{\ud}{{\rm d}}

\newcommand{\st}{\; ; \;}
\newcommand{\Id}[1][d]{\vec{I}_{#1}}

\newcommand{\norm}[2][]{\|#2\|_{#1}}
\newcommand{\seminorm}[2][]{|#2|_{#1}}

\newcommand{\symm}{{\rm s}}

\newcommand{\strain}{\matr{\varepsilon}}

\newcommand{\Real}{\mathbb{R}}
\newcommand{\Natural}{\mathbb{N}}
\newcommand{\Integer}{\mathbb{Z}}

\newcommand{\tenf}{\mathfrak{C}}

\newcommand{\Lvec}[2][{2}]{\textbf{L}^{#1}(#2)}

\newcommand{\Lmats}[2][{2}]{\mathbb{L}^{#1}_\symm(#2)}

\newcommand{\Hvec}[2][{1}]{\textbf{H}^{#1}(#2)}
\newcommand{\HvecD}[2][{1}]{\textbf{H}_0^{#1}(#2)}
\newcommand{\Hmat}[2][{1}]{\mathbb{H}^{#1}(#2)}

\newcommand{\Hmats}[2][{1}]{\mathbb{H}^{#1}_\symm(#2)}

\newcommand{\Pvec}[2][k]{\textbf{P}^{#1}(#2)}
\newcommand{\Pmat}[2][k]{\mathbb{P}^{#1}(#2)}
\newcommand{\Pmats}[2][k]{\mathbb{P}_\symm^{#1}(#2)}

\newcommand{\Mh}[1][h]{\mathcal{M}_{#1}}
\newcommand{\Th}[1][h]{\mathcal{T}_{#1}}
\newcommand{\Fh}[1][h]{\mathcal{F}_{#1}}
\newcommand{\Fhi}{\Fh^{{\rm i}}}
\newcommand{\Fhb}{\Fh^{{\rm b}}}

\newcommand{\normal}{\vec{n}}

\newcommand{\UT}{\underline{\vec{U}}^k_T}
\newcommand{\Uh}{\underline{\vec{U}}^k_h}
\newcommand{\UhD}{\underline{\vec{U}}^k_{h,\mathrm{D}}}
\newcommand{\Ph}[1][k]{P_h^{#1}}

\newcommand{\GTs}{\vec{G}^{k}_{\symm,T}}

\newcommand{\rT}{\vec{r}^{k+1}_T}

\newcommand{\Ih}{\underline{I}^k_h}
\newcommand{\IvT}{\underline{\vec{I}}^k_T}
\newcommand{\Ivh}{\underline{\vec{I}}^k_h}
\newcommand{\lproj}[2][h]{\pi_{#1}^{#2}}
\newcommand{\vlproj}[2][h]{\vec\pi_{#1}^{#2}}

\newcommand{\ms}[1][]{\matr{\sigma}_{#1}}

\newcommand{\oms}[1][]{\overline{\matr{\sigma}}_{#1}}

\newcommand{\vu}[1][]{\vec{u}_{#1}}

\newcommand{\vv}[1][]{\vec{v}_{#1}}
\newcommand{\vw}[1][]{\vec{w}_{#1}}
\newcommand{\vf}[1][]{\vec{f}_{#1}}
\newcommand{\ovu}[1][]{\overline{\vec{u}}_{#1}}
\newcommand{\ovf}[1][]{\overline{\vec{f}}_{#1}}

\newcommand{\uvu}[1][h]{\underline{\vec{u}}_{#1}}
\newcommand{\uvv}[1][h]{\underline{\vec{v}}_{#1}}
\newcommand{\uvw}[1][h]{\underline{\vec{w}}_{#1}}
\newcommand{\uvy}[1][h]{\underline{\vec{y}}_{#1}}
\newcommand{\uvz}[1][h]{\underline{\vec{z}}_{#1}}
\newcommand{\uve}[1][h]{\underline{\vec{e}}_{#1}}

\newcommand{\jump}[2][F]{[#2]_{#1}}
\newcommand{\wavg}[2][F]{\{#2\}_{#1}}
\newcommand{\diff}[1][]{\matr{\kappa}_{#1}}
\newcommand{\udiff}[1][]{\overline{\kappa}_{#1}}
\newcommand{\ldiff}[1][]{\underline{\kappa}_{#1}}
\newcommand{\sdiff}[1][F]{\kappa_{#1}}

\newcommand{\tph}[1][h]{\widehat{p}_h}

\newcommand{\tF}{t_{\rm F}}
\newcommand{\dt}{{\rm d}_t}
\newcommand{\ddt}{\delta_t}

\newcommand{\term}{\mathfrak{T}}

\newtheorem{theorem}{Theorem}
\newtheorem{proposition}[theorem]{Proposition}
\newtheorem{lemma}[theorem]{Lemma}

\theoremstyle{remark}
\newtheorem{remark}[theorem]{Remark}
\theoremstyle{definition}
\newtheorem{assumption}[theorem]{Assumption}

\title{A high-order discretization of nonlinear poroelasticity
\footnote{This work was partially funded by the Bureau de Recherches G\'{e}ologiques et Mini\`{e}res. The work of M. Botti was additionally partially supported by Labex NUMEV (ANR-10-LABX-20) ref. 2014-2-006.
The work of D. A. Di Pietro was additionally partially supported by project HHOMM (ANR-15-CE40-0005).}}
\author[1]{Michele Botti\footnote{\email{michele.botti@polimi.it}}}
\author[2]{Daniele A. Di Pietro\footnote{\email{daniele.di-pietro@umontpellier.fr}}}
\author[3]{Pierre Sochala\footnote{\email{p.sochala@brgm.fr}}}
\affil[1]{Department of Mathematics, Politecnico di Milano, 20133 Milano, Italy}
\affil[2]{IMAG, Université de Montpellier, CNRS, 34095 Montpellier, France}
\affil[3]{Bureau de Recherches Géologiques et Minières, 45060 Orléans, France}

\begin{document}
\maketitle

\begin{abstract}
In this work we construct and analyze a nonconforming high-order discretization method for the quasi-static single-phase nonlinear poroelasticity problem describing Darcean flow in a deformable porous medium saturated by a slightly compressible fluid. The nonlinear elasticity operator is discretized using a Hybrid High-Order method, while the Darcy operator relies on a Symmetric Weighted Interior Penalty discontinuous Galerkin scheme. 
The method is valid in two and three space dimensions, delivers an inf-sup stable discretization on general meshes 
including polyhedral elements and nonmatching interfaces, supports arbitrary approximation orders, and has a reduced cost 
thanks to the possibility of statically condensing a large subset of the unknowns for linearized versions of the problem. 
Moreover, the proposed construction can handle both nonzero and vanishing specific storage coefficients.
\medskip\\
\textbf{Key words.} Nonlinear poroelasticity,
nonlinear Biot problem,
Korn's inequality,
Hybrid High-Order methods,
discontinuous Galerkin methods,
polyhedral meshes
\medskip\\
\textbf{AMS subject classification.} 65N08, 65N30, 76S05
\end{abstract}

\section{Introduction}

In this paper we analyze a Hybrid High-Order (HHO) discretization method for nonlinear poroelastic models. 
We overstep a previous work~\cite{Boffi.Botti.Di-Pietro:16} devoted to the linear Biot model~\cite{Biot:41,Terzaghi:43}  by incorporating more general, possibly nonlinear stress-strain constitutive laws~\cite{FVCA8}.
The model is valid under the assumptions of small deformations of the rock matrix, small variations of the porosity, and 
small relative variations of the fluid density. 
The interest of the poroelastic models considered here is particularly manifest in geosciences applications~\cite{Hu.Winterfield.ea:13,Jin.Zoback:17,Minkoff.Stone.ea:03}, where fluid flows in geological subsurface, modeled as a porous media, induce a deformation of the rock matrix. The challenge is then to design a discretization method able to
(i) treat a complex geometry with polyhedral meshes and nonconforming interfaces, 
(ii) handle possible heterogeneities of the poromechanical parameters and nonlinearities of the stress-strain relation, and 
(iii) deal with the numerical instabilities encountered in this type of coupled problem.

Let $\Omega\subset\Real^d$, $d\in\{2,3\}$, denote a bounded connected polyhedral domain with boundary $\partial\Omega$ and outward normal $\normal$. Without loss of generality, we assume that the domain is scaled so that its diameter is equal to 1.
For a given finite time $\tF>0$, volumetric load $\vf$, fluid source $g$, we consider the nonlinear poroelasticity problem that consists in finding a vector-valued displacement field $\vu$ and a 
scalar-valued pore pressure field $p$ solution of
\begin{subequations}\label{eq:nl_biot.strong}
  \begin{alignat}{2}
    \label{eq:nl_biot.strong.mech}
    -\DIV\ms(\cdot,\GRADs\vu) + \alpha\GRAD p &= \vf &\qquad&\text{in $\Omega\times (0,\tF)$},
    \\
    \label{eq:nl_biot.strong.flow}
    C_0\dt p + \alpha\dt(\DIV\vu) - \DIV(\diff(\cdot)\GRAD p) &= g &\qquad&\text{in $\Omega\times (0,\tF)$},
  \end{alignat}
where $\GRADs$ denotes the symmetric gradient, $\dt$ denotes the time derivative, $\alpha$ is the Biot--Willis coefficient, $C_0\ge 0$ is the constrained specific storage coefficient, and, denoting by $\Real^{d\times d}_\symm$ the set of real-valued, symmetric square matrices, $\diff:\Omega\to\Real^{d\times d}_\symm$ is the uniformly elliptic permeability tensor field which, for real numbers $0<\ldiff\le\udiff$, satisfies for almost every (a.e.) $\vec{x}\in\Omega$ and all $\vec{\xi}\in\Real^d$,
$$
\ldiff\seminorm{\vec\xi}^2\le\diff(\vec{x})\vec{\xi}\cdot\vec{\xi}\le\udiff\seminorm{\vec\xi}^2.
$$
For the sake of simplicity, we assume in the following discussion that $\diff$ is piecewise constant on a polyhedral partition $P_\Omega$ of $\Omega$, an assumption typically verified in geoscience applications.
In the poroelasticity theory \cite{Coussy:04}, the medium is modeled as a continuous superposition of solid and fluid phases. 
The momentum equilibrium equation~\eqref{eq:nl_biot.strong.mech} is based on the Terzaghi decomposition \cite{Terzaghi:43} of the total stress tensor into a mechanical contribution and a pore pressure contribution.
Examples and assumptions for the constitutive stress-strain relation $\ms:\Omega\times\Real^{d\times d}_\symm\to \Real^{d\times d}_\symm$ are detailed in Section \ref{sec:continuous.setting:laws}; we refer the reader to \cite{Bemer.Bouteca:01,Biot:73} for a physical and experimental investigation of the nonlinear behavior of porous solids.
On the other hand, the mass conservation equation~\eqref{eq:nl_biot.strong.flow} is derived for fully saturated porous media assuming Darcean flow. The first two terms of this equation quantify the variation of fluid content in the pores. The dimensionless coupling coefficient $\alpha$ expresses the amount of fluid that can be forced into the medium by a variation of pore volume for a constant fluid pressure, while $C_0$ measures the amount of fluid that can be forced into the medium by pressure increments due to compressibility of the structure. The case of a solid matrix with incompressible grains corresponds to the limit value $C_0=0$.
Following \cite{Showalter:00,Zenisek:84}, for the sake of simplicity we take $\alpha=1$ in what follows.
To close the problem, we enforce homogeneous boundary conditions corresponding to a clamped, impermeable boundary, i.e.,
  \begin{align}
    \label{eq:nl_biot.strong:bc.u}
    \vec{u} &= \vec{0} \qquad\text{on $\partial\Omega\times(0,\tF)$},
    \\
    \label{eq:nl_biot.strong:bc.p}
    (\diff(\cdot)\GRAD p)\cdot\normal &= 0 \qquad\text{on $\partial\Omega\times(0,\tF)$},
  \end{align}
as well as the following initial condition which prescribes the initial fluid content:
\begin{equation}\label{eq:nl_biot.strong:initial}
  C_0 p(\cdot,0)+\DIV\vu(\cdot,0)=\phi^0(\cdot). 
\end{equation}
In the case $C_0=0$, we also need the following compatibility conditions on $g$ and $\phi^0$ and zero-average constraint on $p$:
\begin{equation}\label{eq:compatibility}
  \int_\Omega \phi^0 = 0, \qquad\quad\; \int_\Omega g(\cdot,t) = 0,\quad\text{ and }\quad\int_\Omega p(\cdot,t) = 0
  \qquad\forall t\in(0,\tF).
\end{equation}
\end{subequations}

When discretizing the poroelasticity system~\eqref{eq:nl_biot.strong}, the main challenges are to ensure stability and convergence under mild assumptions on the nonlinear stress-strain relation and on the permeability field, and to prevent localized pressure oscillations arising in the case of poorly permeable, quasi-incompressible porous media. 
Since the latter issue is in part related to the saddle point structure in the coupled equations for $C_0=0$ and small $\ldiff$, the discrete spaces for the displacement and the pressure should satisfy an inf-sup condition. Indeed, as observed in~\cite{Murad.Loula:94,Haga.Osnes.ea:12,Phillips.Wheeler:09} in the context of finite element discretizations of the linear poroelasticity problem, the inf-sup condition yields an $L^2$-estimate of the discrete pressure independent of $\ldiff^{-1}$, and allows one to prove the convergence of the approximate pressure towards the continuous pressure also in the incompressible case $C_0=0$.
We notice, however, that the problem of spurious pressure oscillations is actually more involved than a simple saddle-point coupling issue. For instance, it has been recently pointed out in~\cite{Rodrigo.Gaspar.ea:15} that, even for discretization methods leading to an inf-sup stable discretization of the Stokes problem in the steady case, pressure oscillations can arise owing to a lack of monotonicity of the discrete operator.
The robustness with respect to spurious oscillations has been numerically observed in~\cite[Section 6.2]{Boffi.Botti.Di-Pietro:16} for a HHO--dG discretization of the linear poroelasticity model.

In this work, we present and analyze a nonconforming space discretization of problem~\eqref{eq:nl_biot.strong} where the nonlinear elasticity operator is discretized using the HHO method of~\cite{Botti.Di-Pietro.Sochala:17} (c.f. also~\cite{Di-Pietro.Ern:15,Di-Pietro.Drouniou:15}), while the Darcy operator relies on the Symmetric Weighted Interior Penalty (SWIP) method of~\cite{Di-Pietro.Ern.ea:08}.
The proposed method has several assets:
\begin{inparaenum}[(i)]
\item it is valid in two and three space dimensions;
\item it delivers an inf-sup stable discretization on general spatial meshes including, e.g., polyhedral elements and nonmatching interfaces;
\item it allows one to increase the space approximation order to accelerate convergence in the presence of (locally) regular solutions.
\end{inparaenum}
Compared to the method proposed in~\cite{Boffi.Botti.Di-Pietro:16} for the linear poroelasticity problem, there are two main differences in the design.
First, for a given polynomial degree $k\ge 1$, the symmetric gradient reconstruction sits in the full space of tensor-valued polynomials of total degree $\le k$, as opposed to symmetric gradients of vector-valued polynomials of total degree $\le (k+1)$. Following~\cite{Di-Pietro.Droniou.ea:18,Botti.Di-Pietro.Sochala:17}, this modification is required to obtain optimal convergence rates when considering nonlinear stress-strain laws. Second, the right-hand side of the discrete problem of Section \ref{sec:disc.pb} is obtained by taking the average in time of the loading force $\vf$ and fluid source $g$ over a time step instead of their value at the end of the time step. This modification allows us to prove stability and optimal error estimates under significantly weaker time regularity assumptions on data (cf. Remarks \ref{rem:time_reg1} and \ref{rem:time_reg2}). 
Finally, in Section \ref{sec:discrete.spaces} we give a new simple proof of a discrete counterpart of Korn's inequality on HHO spaces, not requiring particular geometrical assumptions on the mesh.
The interest of these results goes beyond the specific application considered here.

The material is organized as follows.
In Section~\ref{sec:continuous_setting} we present the assumptions on the stress-strain law and the variational formulation of the nonlinear poroelasticity problem. 
In Section~\ref{sec:discrete_setting} we define the space and time meshes and the discrete spaces for the displacement and the pressure fields.  
In Section~\ref{sec:discretization} we define the discrete counterparts of the elasticity, Darcy, and hydromechanical coupling operators and formulate the discrete problem.
In Section~\ref{sec:stability} we prove the well-posedness of the scheme by deriving an a priori estimate on the discrete solution that holds also when the specific storage coefficient vanishes. 
The convergence analysis of the method is carried out in Section~\ref{sec:convergence}. 
Finally, Section~\ref{sec:num.res} contains numerical tests to asses the performance of the method.


\section{Continuous setting} \label{sec:continuous_setting}

In this section we introduce the notation for function spaces, formulate the assumptions on the stress-strain law, and derive a weak formulation of problem \eqref{eq:nl_biot.strong}.

\subsection{Notation for function spaces}

Let $X\subset\overline{\Omega}$.
Spaces of functions, vector fields, and tensor fields defined over $X$ are respectively denoted by italic capital, boldface Roman capital, and special Roman capital letters.
The subscript ``s'' appended to a special Roman capital letter denotes a space of symmetric tensor fields.
Thus, for example, $L^2(X), \Lvec{X}$, and $\Lmats{X}$ respectively denote the spaces of square integrable functions, vector fields, and symmetric tensor fields over $X$.
For any measured set $X$ and any $m\in\Integer$, we denote by $H^m(X)$ the usual Sobolev space of functions that have weak partial derivatives of order up to $m$ in $L^2(X)$, with the convention that $H^0(X)\coloneq L^2(X)$, while $C^m(X)$ and $C_{\rm c}^\infty(X)$ denote, respectively, the usual spaces of $m$-times continuously differentiable functions and infinitely continuously differentiable functions with compact support on $X$. We denote by $(\cdot,\cdot)_X$ and $(\cdot,\cdot)_{m,X}$ the usual scalar products in $L^2(X)$ and $H^{m}(X)$ respectively, and by $\norm[X]{{\cdot}}$ and $\norm[m,X]{{\cdot}}$ the induced norms.

For a vector space $V$ with scalar product $(\cdot,\cdot)_V$, the space $C^{m}(V)\coloneq C^{m}([0,\tF];V)$ is spanned by $V$-valued functions that are $m$-times continuously differentiable in the time interval $[0,\tF]$. The space $C^{m}(V)$ is a Banach space when equipped with the norm 
$$
\norm[C^{m}(V)]{\varphi}\coloneq\max_{0\le i\le m}\max_{t\in[0,\tF]}\norm[V]{\dt^{i}\varphi(t)}.
$$ 
Similarly, the Hilbert space $H^{m}(V)\coloneq H^{m}((0,\tF);V)$ is spanned by $V$-valued functions of the time interval,  and the norm $\norm[H^{m}(V)]{{\cdot}}$ is induced by the scalar product
$$
(\varphi,\psi)_{H^{m}(V)}=\sum_{j=0}^m \int_0^{\tF} (\dt^{j}\varphi(t),\dt^{j}\psi(t))_V \ud t
\qquad \forall\varphi,\psi \in H^{m}(V).
$$

\subsection{Stress-strain law}\label{sec:continuous.setting:laws}

The following assumptions on the stress-strain relation are required to obtain a well-posed weak formulation of the nonlinear poroelasticity problem.
\begin{assumption}[Stress-strain relation]
  \label{ass:hypo}
We assume that the stress function $\ms:\Omega\times\Real^{d\times d}_\symm\to\Real^{d\times d}_\symm$ is a 
Carath\'eodory function, i.e., $\ms(\vec{x},\cdot)$ is continuous on $\Real^{d\times d}_\symm$ for almost every $\vec{x}\in\Omega$ and $\ms(\cdot,\matr{\tau})$ is measurable on $\Omega$ for all $\matr{\tau}\in\Real^{d\times d}_\symm$.
Moreover, there exist real numbers $C_{\rm gr},C_{\rm cv}\in(0,+\infty)$ such that, for a.e. 
$\vec{x}\in\Omega$ and all $\matr{\tau},\matr{\eta}\in\Real^{d\times d}_\symm$, the following conditions hold:
\begin{subequations}\label{eq:hypo}
  \begin{align} 
  \label{eq:hypo.growth}
  \seminorm[d\times d]{\ms(\vec{x},\matr{\tau})} &\le C_{\rm gr} \seminorm[d\times d]{\matr{\tau}},
  &&\text{(growth)}
  \\
  \label{eq:hypo.coercivity}
  \ms(\vec{x},\matr{\tau}) :\matr{\tau} &\ge C_{\rm cv}^2 \seminorm[d\times d]{\matr{\tau}}^2,
  &&\text{(coercivity)}
  \\
  \label{eq:hypo.monotonicity}
  \left(\ms(\vec{x},\matr{\tau})-\ms(\vec{x},\matr{\eta})\right):\left(\matr{\tau}-\matr{\eta}\right) &> 0 \mbox{ if }\matr{\eta}\neq\matr{\tau}.
  &&\text{(monotonicity)}
  \end{align}
\end{subequations}
Above, we have introduced the Frobenius product such that, for all $\matr{\tau},\matr{\eta}\in\Real^{d\times d}$, $\matr{\tau}:\matr{\eta}\coloneq\sum_{1\le i,j\le d}\tau_{ij}\eta_{ij}$ with corresponding matrix norm such that, for all $\matr{\tau}\in\Real^{d\times d}$, $\seminorm[d\times d]{\matr{\tau}}\coloneq(\matr{\tau}:\matr{\tau})^{\frac12}$.
\end{assumption}
Three meaningful examples for the stress-strain relation $\ms:\Omega\times\Real^{d\times d}_\symm\to \Real^{d\times d}_\symm$ in \eqref{eq:nl_biot.strong.mech} are: 
\begin{itemize}
  \item The (possibly heterogeneous) \textit{linear elasticity model} given by the usual Hooke's law
  \begin{equation}
    \label{eq:LinCauchy}
    \ms(\cdot,\matr{\tau})=\lambda(\cdot)\optr(\matr{\tau})\Id+2\mu(\cdot)\matr{\tau},
  \end{equation}
  where $\mu:\Omega\to [\mu_*,\mu^*]$, with $0<\mu_*\le\mu^*<+\infty$, 
  and $\lambda:\Omega\to\Real_+$ are the Lam\'e parameters.
  \item The \textit{nonlinear Hencky--Mises model} of~\cite{Necas:86,Gatica.Stephan:02} corresponding to the mechanical behavior law
  \begin{equation}
    \label{eq:HenckyMises}
    \ms(\cdot,\matr{\tau})=\tilde{\lambda}(\cdot,\opdev(\matr{\tau}))\optr (\matr{\tau})\Id
    +2\tilde{\mu}(\cdot,\opdev(\matr{\tau}))\matr{\tau},
  \end{equation}  
  with nonlinear Lam\'e scalar functions $\tilde\mu:\Omega\times\Real_+\to[\mu_*,\mu^*]$ and 
  $\tilde\lambda:\Omega\times\Real_+\to\Real_+$ depending on the deviatoric part 
  $\opdev(\matr{\tau})\coloneq\optr (\matr{\tau}^2)-\frac1d\optr (\matr{\tau})^2$ of the strain.
 \item The \textit{isotropic reversible hyperelastic damage model}~\cite{Cervera.Chiumenti.ea:10}, for which the stress-strain relation reads
  \begin{equation}
    \label{eq:Damage}
    \ms(\cdot,\matr{\tau})=(1-D(\cdot,\matr{\tau}))\tenf(\cdot)\matr{\tau}.
  \end{equation}
  where $D:\Omega\times\Real^{d\times d}_\symm\to[0,1]$ is the scalar damage function and 
  $\tenf:\Omega\to\Real^{d^4}$ is a fourth-order symmetric and uniformly elliptic tensor field, namely, for some 
  strictly positive constants $\underline{C}$ and $\overline{C}$, it holds
  $$
    \underline{C}\seminorm[d\times d]{\matr{\tau}}^2\le 
    \tenf(\vec{x})\matr{\tau}:\matr{\tau}\le \overline{C}\seminorm[d\times d]{\matr{\tau}}^2
    \quad\forall\matr{\tau}\in\Real^{d\times d},\;\forall\vec{x}\in\Omega.
  $$
\end{itemize}

Being linear, the Cauchy stress tensor in \eqref{eq:LinCauchy} clearly satisfies the previous assumptions.
Moreover, under some mild requirements (cf. \cite{Botti.Riedlbeck:18,Droniou.Lamichhane:15}) on the nonlinear Lam\'e 
scalar functions $\tilde\mu$ and $\tilde\lambda$ in \eqref{eq:HenckyMises} and on the damage function $D$ 
in \eqref{eq:Damage}, it can be proven that also the Hencky--Mises model and the isotropic reversible damage model satisfy Assumption~\ref{ass:hypo}.

\subsection{Weak formulation}\label{sec:weak_form}

At each time $t\in [0,\tF]$, the natural functional spaces for the displacement $\vu(t)$ and pore pressure $p(t)$ taking into account the boundary condition \eqref{eq:nl_biot.strong:bc.u} and the zero average constraint \eqref{eq:compatibility} are, respectively, 
$$
  \vec{U}\coloneq\HvecD{\Omega}
  \qquad\text{and}\qquad
  P\coloneq\begin{cases}
  {H}^1(\Omega) & \text{if $C_0>0$,} \\ 
  {H}^1(\Omega)\cap L^2_0(\Omega)  & \text{if $C_0=0$,}
  \end{cases}
$$
with $\HvecD{\Omega}\coloneq\left\{\vv\in\Hvec{\Omega}\st {\vv}_{|\partial\Omega}=\vec{0}\right\}$ and $L^2_0(\Omega)\coloneq\left\{q\in L^2(\Omega)\st \int_\Omega q=0\right\}$.
We consider the following weak formulation of problem~\eqref{eq:nl_biot.strong}:
For a loading term $\vf\in L^2(\Lvec{\Omega})$, a fluid source $g\in L^2(L^2(\Omega))$, and an initial datum $\phi^0\in L^2(\Omega)$ that verify \eqref{eq:compatibility} if $C_0>0$, find $\vu\in L^2(\vec{U})$ and $p\in L^2(P)$ such that, for all $\vv\in\vec{U}$, all $q\in P$, and all $\varphi\in C_{\rm c}^\infty((0,\tF))$
\begin{subequations}
  \label{eq:weak_form}
  \begin{align}
    \label{eq:weak_form.mech}
    \int_0^{\tF}a(\vu(t),\vv)\,\varphi(t) \ud t+\int_0^{\tF}\hspace{-1mm}b(\vv,p(t))\,\varphi(t)\ud t
    &=\int_0^{\tF}(\vf(t),\vv)_\Omega\,\varphi(t)\ud t,
    \\
    \label{eq:weak_form.fluid}
    \int_0^{\tF}\left[b(\vu(t),q)-C_0 (p(t),q)_{\Omega}\right] \dt\varphi(t) \ud t 
    +\int_0^{\tF}c(p(t),q)\,\varphi(t) \ud t
    &=\int_0^{\tF}(g(t),q)_\Omega\,\varphi(t) \ud t, 
    \\
    \label{eq:weak_form.initial}
    (C_0 p(0) + \DIV\vu(0),q)_\Omega &= (\phi^0,q)_\Omega,
  \end{align}
\end{subequations}
where we have defined the nonlinear function $a:\vec{U}\times\vec{U}\to\Real$ and the bilinear forms $b:\vec{U}\times P\to\Real$ and $c:P\times P\to\Real$ such that, for all $\vv,\vw\in\vec{U}$ and all $q,r\in P$,
$$
a(\vv,\vw)\coloneq (\ms(\cdot,\GRADs\vv),\GRADs\vw)_\Omega,\qquad
b(\vv,q)\coloneq -(\DIV\vv,q)_\Omega,\qquad
c(q,r)\coloneq (\diff(\cdot)\GRAD r, \GRAD q)_\Omega.
$$
The first term in \eqref{eq:weak_form.mech} is well defined thanks to the growth assumption \eqref{eq:hypo.growth}. Moreover, owing to \eqref{eq:hypo.coercivity} together with Korn's first inequality and Poincaré's inequality, 
$a(\cdot,\cdot)$ and $c(\cdot,\cdot)$ are coercive on $\vec{U}$ and $P$, respectively. The strict monotonicity assumption \eqref{eq:hypo.monotonicity} guarantees the uniqueness of the weak solution.
\begin{remark}[Regularity of the fluid content and of the pore pressure]
  \label{rem:reg_porosity.press_average}
  Using an integration by parts in time in \eqref{eq:weak_form.fluid}, it is inferred that
\begin{equation}
  \label{eq:reg_poro}
\dt\left[C_0 (p,q)_{\Omega}-b(\vu,q)\right]+c(p,q) = (g,q)_\Omega\quad\forall q\in P\quad\text{ in } L^2((0,\tF)). 
\end{equation}
Therefore, defining the fluid content $\phi\coloneq C_0 p + \DIV\vu$, we have that $t\mapsto(\phi(t),q)_\Omega\in H^1{((0,\tF))}\subset C^0([0,\tF])$ for all $q\in P$, and, as a result, \eqref{eq:weak_form.initial} makes sense. Moreover, in the case $C_0>0$, taking $q=1$ in \eqref{eq:reg_poro} and owing to the definition of the bilinear form $c$ and the homogeneous Dirichlet condition \eqref{eq:nl_biot.strong:bc.u}, we infer that
$$
\dt \left(C_0\int_\Omega p(\cdot,t)\right) = \int_\Omega g(\cdot,t)\quad\text{ in } L^2((0,\tF)).
$$
Thus, $t\mapsto\int_\Omega p(\cdot,t)\in H^1{((0,\tF))}$, namely the average of the pore pressure over $\Omega$ is a continuous function in $[0,\tF]$. 
\end{remark}
\section{Discrete setting}\label{sec:discrete_setting}

In this section we define the space and time meshes, recall the definition and properties of $L^2$-orthogonal projectors on local and broken polynomial spaces, and introduce the discrete spaces for the displacement and the pressure.

\subsection{Space mesh}

We consider here polygonal or polyhedral meshes corresponding to couples $\Mh\coloneq(\Th,\Fh)$, where $\Th$ is a finite collection of polygonal elements such that $h\coloneq\max_{T\in\Th}h_T>0$ with $h_T$ denoting the diameter of $T$, while $\Fh$ is a finite collection of hyperplanar faces. It is assumed henceforth that the mesh $\Mh$ matches the geometrical requirements detailed in \cite[Definition 7.2]{Droniou.Eymard.ea:17}; see also~\cite[Section 2]{Di-Pietro.Tittarelli:17}.
To avoid dealing with jumps of the permeability coefficient inside elements, we additionally assume that $\Mh$ is compliant with the partition $P_\Omega$ on which $\diff$ is piecewise constant meaning that, for every $T\in\Th$, there exists a unique subdomain $\omega\in P_\Omega$ such that $T\subset\omega$. 
For every mesh element $T\in\Th$, we denote by $\Fh[T]$ the subset of $\Fh$ containing the faces that lie on the boundary $\partial T$ of $T$.
For each face $F\in\Fh[T]$, $\normal_{TF}$ is the (constant) unit normal
vector to $F$ pointing out of $T$.
Boundary faces lying on $\partial\Omega$ and internal faces contained in
$\Omega$ are collected in the sets $\Fhb$ and $\Fhi$, respectively.

Our focus is on the so-called $h$-convergence analysis, so we consider a sequence of refined meshes that is regular in the sense of~\cite[Definition~3]{Di-Pietro.Tittarelli:17}.
The mesh regularity assumption implies, in particular, that the diameter $h_T$ of a mesh element $T\in\Th$ is uniformly comparable to the diameter $h_F$ of each face $F\in\Fh[T]$, and that the number of faces in $\Fh[T]$ is bounded above by an integer $N_\partial$ independent of $h$.
We additionally assume that, for each mesh in the sequence, all the elements are star-shaped with respect to every point of a ball of radius uniformly comparable to the diameter of the element.
This assumption is required to use the results of \cite[Appendix A]{Botti.Di-Pietro.ea:18}.

\subsection{Time mesh}\label{sec:time_mesh}

We subdivide $(0,\tF)$ into $N\in\Natural^*$ uniform subintervals, and introduce the timestep $\tau\coloneq\tF/N$ and the 
discrete times $t^n\coloneq n\tau$ for all $0\le n\le N$. 
We define the space of piecewise $H^1$ functions on $(0,\tF)$ by
$$
H^1(\mathcal{T}_{\tau})\coloneq 
\left\{ \varphi\in L^2((0,\tF)) \st\varphi_{|(t^n,t^{n-1})}\in H^1((t^{n-1},t^n)) \,\text{ for all }\ 1\le n\le N\right\}.
$$
Since each $\psi\in H^1((t^{n-1},t^n))$ has an absolutely continuous representative in $[t^{n-1},t^n]$, we can identify 
$\varphi\in H^1(\mathcal{T}_{\tau})$ with a left continuous function in $(0,\tF)$.
Therefore, for any vector space $V$ and any $\varphi\in H^1(\mathcal{T}_{\tau};V)$, we set $\varphi^0\coloneq\varphi(0)$ and, for all $1\le n\le N$,
$$
\varphi^n\coloneq\lim_{t\to (t^n)^-}\varphi(t)\in V.
$$
If $\varphi\in C^{0}(V)$, this simply amounts to setting $\varphi^n\coloneq\varphi(t^n)$.
For all $n\ge 1$ and $\psi\in L^1(V)$, we define the time average of $\psi$ in $(t^{n-1}, t^n)$ as
\begin{equation}\label{eq:time.average}
  \overline\psi^n\coloneq \tau^{-1}\int_{t^{n-1}}^{t^n}\psi(t) {\rm d}t \in V,
\end{equation}
with the convention that $\overline\psi^0=0\in V$. 
We also let, for all $(\varphi^i)_{0\le i\le N}\in V^{N+1}$ and all $1\le n\le N$,
$$
\label{eq:ddt}
\ddt\varphi^n\coloneq\frac{\varphi^n-\varphi^{n-1}}{\tau}\in V
$$
denote the backward approximation of the first derivative of $\varphi$ at time $t^n$. 

We note a preliminary result that will be used in the convergence analysis of Section \ref{sec:convergence}. 
Let $\psi\in H^1(\mathcal{T}_{\tau})$ and $1\le n\le N$.
Identifying $\psi_{|(t^{n-1},t^n)}\in H^1((t^{n-1},t^n))$ with its absolutely continuous representative 
in $(t^{n-1},t^n]$, one can use the fundamental theorem of calculus to infer that
$$
\psi^n-\overline{\psi}^n = 
\psi(t^n) - \frac1\tau\int_{t^{n-1}}^{t^n}\left(
\psi(t^n)-\int_s^{t^n}\hspace{-2mm}\dt\psi(t) {\rm d}t
\right) {\rm d}s =
\frac1\tau\int_{t^{n-1}}^{t^n}\int_s^{t^n}\hspace{-2mm}\dt\psi(t) {\rm d}t~{\rm d}s
\le \int_{t^{n-1}}^{t^n} |\dt\psi(t)| {\rm d}t. 
$$
Thus, applying the previous result together with the Jensen inequality yields, 
for all $\varphi\in H^1(\mathcal{T}_{\tau}; L^2(\Omega))$,
\begin{equation}
  \label{eq:time_approx}
  \norm[\Omega]{\varphi^n-\overline{\varphi}^n}^2
  \le\int_\Omega\left(\int_{t^{n-1}}^{t^n}|\dt\varphi(\vec{x},t)| {\rm d}t\right)^2 {\rm d}\vec{x}
  \le\tau\int_{t^{n-1}}^{t^n}\norm[\Omega]{\dt\varphi(t)}^2 {\rm d}t
  \le\tau\norm[H^1((t^{n-1},t^n);L^2(\Omega))]{\varphi}^2,
\end{equation}
where $\varphi(\vec{x},t)$ is a shorthand notation for $(\varphi(t))(\vec{x})$. As a result of \eqref{eq:time_approx}, we get 
$$
\sum_{n=1}^N \tau \norm[\Omega]{\varphi^n-\overline{\varphi}^n}^2 \le 
\tau^2 \sum_{n=1}^N\norm[H^1((t^{n-1},t^n);L^2(\Omega))]{\varphi}^2 
\eqcolon \tau^2 \norm[H^1(\mathcal{T}_{\tau}; L^2(\Omega))]{\varphi}^2.
$$

\subsection{$L^2$-orthogonal projectors on local and broken polynomial spaces}

For $X\subset\overline{\Omega}$ and $k\in\Natural$, we denote by $P^{k}(X)$ the space spanned by the restriction to $X$ of scalar-valued, $d$-variate polynomials of total degree $k$.
The $L^2$-projector $\lproj[X]{k}:L^1(X)\to P^{k}(X)$ is defined such that, for all $v\in L^1(X)$,
\begin{equation}\label{eq:lproj}
  \int_X(\lproj[X]{k}v-v)w=0\qquad\forall w\in P^{k}(X).
\end{equation}
As a projector, $\lproj[X]{k}$ is linear and idempotent so that, for all $v\in P^k(X)$, $\lproj[X]{k}v=v$.
When dealing with the vector-valued polynomial space $\Pvec[k]{X}$ or with the tensor-valued polynomial space 
$\Pmat[k]{X}$, we use the boldface notation $\vlproj[X]{k}$ for the corresponding $L^2$-orthogonal 
projectors acting component-wise.
At the global level, we denote by $P^{k}(\Th)$, $\Pvec[k]{\Th}$, and $\Pmat[k]{\Th}$, respectively, the spaces of scalar-valued, vector-valued, and tensor-valued broken polynomial functions on $\Th$ of total degree $\le k$, and by $\lproj[h]{k}$ and $\vlproj[h]{k}$ the $L^2$-projectors on $P^{k}(\Th)$ and $\Pvec[k]{\Th}$, respectively.
The following optimal approximation properties for the $L^2$-projector $\lproj[X]{k}$ follow from \cite[Lemmas 3.4 and 3.6]{Di-Pietro.Droniou:15}:
There exists a strictly positive real number $C_{\rm ap}$ independent of $h$ such that, for all $T\in\Th$, all $l\in\{0,\ldots,k+1\}$, all $m\in\{0,\ldots,l\}$, and all $v\in H^l(T)$,
\begin{equation}
  \label{eq:approx.lproj}
  \seminorm[H^m(T)]{v - \lproj[T]{k} v }
  \le 
  C_{\rm ap} h_T^{l-m} \seminorm[H^l(T)]{v}
\end{equation}
and, if $l\ge 1$ and $m\le l-1$,
\begin{equation}
  \label{eq:approx.lproj:trace}
  h_T^{\frac12}\seminorm[{H^m(\Fh[T])}]{v - \lproj[T]{k} v }
  \le 
  C_{\rm ap} h_T^{l-m} \seminorm[H^l(T)]{v}
\end{equation}
where $\seminorm[{H^m(\Fh[T])}]{{\cdot}}$ is the broken Sobolev seminorm on $\Fh[T]$.

\subsection{Discrete spaces}\label{sec:discrete.spaces}

In this section we define the discrete spaces upon which the HHO method corresponding to a polynomial degree $k\ge 1$ is built.

\subsubsection{Displacement}

The discrete unknowns for the displacement are collected in the space
$$
\Uh\coloneq\left\{
\uvv[h]=\big((\vv[T])_{T\in\Th},(\vv[F])_{F\in\Fh}\big)\st
\vv[T]\in\Pvec[k]{T}\mbox{ for all } T\in\Th
\mbox{ and }\vv[F]\in\Pvec[k]{F}\mbox{ for all }F\in\Fh
\right\}.
$$
For any $\uvv[h]\in\Uh$, we denote by $\vv[h]\in\Pvec[k]{\Th}$ the broken polynomial vector field obtained patching element-based unknowns, so that
$$
(\vv[h])_{|T}=\vv[T]\qquad\forall T\in\Th.
$$
The discrete unknowns corresponding to a function $\vv\in\Hvec{\Omega}$ are obtained by means of the interpolator $\Ivh:\Hvec{\Omega}\to\Uh$ such that
\begin{equation}\label{eq:Ivh}
  \Ivh\vv \coloneq \big( (\vlproj[T]{k}\vv[|T])_{T\in\Th}, (\vlproj[F]{k}\vv[|F])_{F\in\Fh} \big).
 \end{equation}
For all $T\in\Th$, we denote by $\UT$ and $\IvT$ the restrictions to $T$ of $\Uh$ and $\Ivh$, respectively, and, for any $\uvv[h]\in\Uh$, we let $\uvv[T]\coloneq\big(\vv[T],(\vv[F])_{F\in\Fh[T]}\big)$ collect the local discrete unknowns attached to $T$.
At each time step, the displacement is sought in the following subspace of $\Uh$ that strongly accounts for the homogeneous Dirichlet condition~\eqref{eq:nl_biot.strong:bc.u}:
$$
\UhD\coloneq\left\{
\uvv[h]=\big((\vv[T])_{T\in\Th},(\vv[F])_{F\in\Fh}\big)\in\Uh\st
\vv[F]=\vec{0}\quad\forall F\in\Fhb\right\}.
$$
We next prove a discrete version of Korn's first inequality on $\UhD$ that will play a key role in the analysis.
To this purpose, we endow the space $\Uh$ with the discrete strain seminorm $\norm[\strain,h]{{\cdot}}$ defined, for all $\uvv[h]\in\Uh$, such that
\begin{equation}\label{eq:strain.seminorm}
  \norm[\strain,h]{\uvv[h]}\coloneq\left[
    \sum_{T\in\Th}\left(\norm[T]{\GRADs\vv[T]}^2
    + \sum_{F\in\Fh[T]}h_F^{-1}\norm[F]{\vv[F]-\vv[T]}^2\right)
    \right]^{\frac12}.
\end{equation}
We will also need the following continuous trace inequality, whose proof follows the arguments of \cite[Lemma 1.49]{Di-Pietro.Ern:12} (where a slightly different notion of mesh faces is considered):
There exists a strictly positive real number $C_{\rm tr}$, independent of $h$ but possibly depending on the mesh regularity parameter, such that, for all $T\in\Th$, all $\vv[T]\in\Hvec{T}$, and all $F\in\Fh[T]$,
\begin{equation}
  \label{eq:continuous_trace}
  \norm[F]{\vv[T]}^2\le 
  C_{\rm tr}^2 \left(\norm[T]{\GRAD\vv[T]}+h_T^{-1}\norm[T]{\vv[T]}\right)\norm[T]{\vv[T]}.
\end{equation}
\begin{proposition}[Discrete Korn's first inequality]
  There is a real number $C_{\rm K}>0$, only depending on $\Omega$, $d$, and the mesh regularity parameter, such that, 
  for all $\uvv[h]\in\UhD$,
  \begin{equation}
    \label{eq:Korn}
    \norm[\Omega]{\vv[h]}\le C_{\rm K}\norm[\strain,h]{\uvv[h]}.
  \end{equation}
\end{proposition}
\begin{remark}[Strain norm]
  An immediate consequence of \eqref{eq:Korn} is that the map $\norm[\strain,h]{{\cdot}}$ is a norm in $\UhD$.
\end{remark}
\begin{proof}
  We start by noticing that it holds, for all $\matr{\alpha}\in\Real^{d\times d}_\symm$ and all $\matr{\beta}\in\Real^{d\times d}$, denoting by $\matr{\beta}_\symm\coloneq\frac12(\matr{\beta}+\matr{\beta}\trans)$ the symmetric part of $\matr{\beta}$,
  \begin{equation}\label{eq:magic}
    \matr{\alpha}:\matr{\beta}
    =\matr{\alpha}:\matr{\beta}_\symm.
  \end{equation}
  Let $\uvv[h]\in\UhD$. Since the divergence operator $\DIV:\Hmats{\Omega}\to\Lvec{\Omega}$ is onto 
  (c.f.~\cite[Section 9.1.1]{Boffi.Brezzi.ea:13} and \cite[Theorem 3.2]{Amrouche.al:06}), there exists $\matr{\tau}_{\vv[h]}\in\Hmats{\Omega}$ such that 
  $\DIV\matr{\tau}_{\vv[h]}=\vv[h]$ and $\norm[{1,\Omega}]{\matr{\tau}_{\vv[h]}}\le C_{\rm{sj}}\norm[\Omega]{\vv[h]}$, 
  with $C_{\rm{sj}}>0$ independent of $h$. 
  It follows that
  $$
  \begin{aligned}
    \norm[\Omega]{\vv[h]}^2
    &=\int_\Omega\vv[h]\cdot\left(\DIV\matr{\tau}_{\vv[h]}\right)
    \\
    &=\sum_{T\in\Th}\left(-\int_T\GRAD\vv[T]:\matr{\tau}_{\vv[h]}
      +\sum_{F\in\Fh[T]}\int_F (\vv[T]-\vv[F])\cdot(\matr{\tau}_{\vv[h]}\normal_{TF})\right)
    \\
    &=\sum_{T\in\Th}\left(-\int_T\GRADs\vv[T]:\matr{\tau}_{\vv[h]}
    +\sum_{F\in\Fh[T]}\int_F (\vv[T]-\vv[F])\cdot(\matr{\tau}_{\vv[h]}\normal_{TF})\right),
  \end{aligned}
  $$
  where, to pass to the second line, we have integrated by parts element by element and used the fact 
  that $\matr{\tau}_{\vv[h]}$ has continuous normal traces across interfaces and that boundary unknowns are set to zero 
  in order to insert $\vv[F]$ into the boundary term, while, to pass to the third line, we have used \eqref{eq:magic} 
  with $\matr{\alpha}=\matr{\tau}_{\vv[h]}$ and $\matr{\beta}=\GRAD\vv[T]$. Applying a Cauchy--Schwarz inequality on 
  the integrals over the element, a generalized H\"older inequality with 
  exponents $(2,2,+\infty)$ on the integrals over faces, using the fact that 
  $\norm[{\Lvec[\infty]{F}}]{\normal_{TF}}\le 1$, and invoking a discrete Cauchy--Schwarz inequality on the sum over 
  $T\in\Th$, we infer that
  $$
  \begin{aligned}
    \norm[\Omega]{\vv[h]}^2
    &\le\sum_{T\in\Th}\left(
    \norm[T]{\GRADs\vv[T]}\norm[T]{\matr{\tau}_{\vv[h]}}
    +\sum_{F\in\Fh[T]}h_F^{-\frac12}\norm[F]{\vv[F]-\vv[T]}~h_F^{\frac12}\norm[F]{\matr{\tau}_{\vv[h]}}
    \right)
    \\
    &\le 
    \left(\sum_{T\in\Th}\norm[T]{\GRADs\vv[T]}^2\right)^{\frac12}
    \norm[\Omega]{\matr{\tau}_{\vv[h]}}
    +\left(\sum_{T\in\Th}\sum_{F\in\Fh[T]}h_F^{-1}\norm[F]{\vv[F]-\vv[T]}^2\right)^{\frac12}
    \left(\sum_{T\in\Th}\sum_{F\in\Fh[T]}h_F\norm[F]{\matr{\tau}_{\vv[h]}}^2\right)^{\frac12}
    \\
    &\le\sqrt{2 N_\partial} C_{\rm tr}
    \left(\sum_{T\in\Th}\norm[T]{\GRADs\vv}^2
    +\sum_{T\in\Th}\sum_{F\in\Fh[T]}h_F^{-1}\norm[F]{\vv[F]-\vv[T]}^2\right)^{\frac12}
    \norm[{1,\Omega}]{\matr{\tau}_{\vv[h]}}
    \\
    &=\sqrt{2N_\partial} C_{\rm tr}\norm[\strain,h]{\uvv[h]}\norm[{1,\Omega}]{\matr{\tau}_{\vv[h]}},
  \end{aligned}
  $$
  where, to pass to the third line, we have estimated $\norm[F]{\matr{\tau}_{\vv[h]}}$ using the continuous trace inequality \eqref{eq:continuous_trace} and used the fact that $h_F\le h_T\le {\rm diam}(\Omega)=1$ for any $T\in\Th$ and $F\in\Fh[T]$.
  Thus, invoking the boundedness of the divergence operator, we get
  $$
  \norm[\Omega]{\vv[h]}^2
  \le \sqrt{2N_\partial} C_{\rm sj}C_{\rm tr}\norm[\strain,h]{\uvv[h]}\norm[\Omega]{\vv[h]},
  $$
  which yields the conclusion with $C_{\rm K}= \sqrt{2N_\partial} C_{\rm sj}C_{\rm tr}$.
\end{proof}

\subsubsection{Pore pressure}

At each time step, the discrete pore pressure is sought in the space
$$
  \Ph\coloneq
    \begin{cases}
      P^k(\Th) & \text{if $C_0>0$},
      \\
      P^k_0(\Th)\coloneq\left\{q_h\in P^k(\Th)\st\int_\Omega q_h=0\right\} & \text{if $C_0=0$.}
    \end{cases}
$$
For any internal face $F\in\Fhi$, we denote by $T_{F,1}, T_{F,2}\in\Th$ the two mesh elements that share $F$, that is to say $F\subset \partial T_{F,1}\cap \partial T_{F,2}$ and $T_{F,1}\neq T_{F,2}$ (the ordering of the elements is arbitrary but fixed), and we set
\begin{equation}
  \label{eq:diff_coeff}
  \sdiff[F,i]\coloneq\left({\diff}_{|T_{F,i}}\normal_{T_{F,i}F}\right)\cdot\normal_{T_{F,i}F}\quad\text{for } i\in\{1,2\},
  \qquad\sdiff\coloneq \frac{2\sdiff[F,1]\sdiff[F,2]}{\sdiff[F,1]+\sdiff[F,2]}.
\end{equation}
For all $q_h\in\Ph$, we denote by $q_T$ the restriction of $q_h$ to an element $T\in\Th$ and we define the discrete seminorm 
\begin{equation}\label{eq:p.seminorm}
  \norm[\diff,h]{q_h}\coloneq\left(
    \sum_{T\in\Th}\norm[T]{\diff^{\frac12}\GRAD q_T}^2
    + \sum_{F\in\Fhi}\frac{\sdiff}{h_F} \norm[F]{q_{T_{F,1}} - q_{T_{F,2}}}^2
    \right)^{\frac12}.
\end{equation}
The fact that in \eqref{eq:p.seminorm} boundary terms only appear on internal faces reflects the homogeneous Neumann boundary condition \eqref{eq:nl_biot.strong:bc.p}. 

Using the surjectivity of the divergence operator $\DIV:\HvecD{\Omega}\to L^2_0(\Omega)$ and proceeding as in the proof of the discrete Korn inequality \eqref{eq:Korn}, a discrete Poincaré--Wirtinger inequality in $P^k(\Th)$ is readily inferred, namely one has the existence of $C_{\rm P}>0$, only depending on $\Omega$, $d$, and the mesh regularity parameter such that, for all 
$q_h\in P^k(\Th)$,
$$
  \norm[\Omega]{q_h-\lproj[\Omega]{0}q_h}\le C_{\rm P}\ldiff^{-\frac12}\norm[\diff,h]{q_h}.
$$
This result ensures, in particular, that the seminorm $\norm[\diff,h]{{\cdot}}$ defined in \eqref{eq:p.seminorm} is a norm on $P^k_0(\Th)$.
For a proof of more general Sobolev inequalities on broken polynomial spaces, we refer the reader to \cite{Di-Pietro.Ern:10} and \cite[Section 5.1.2]{Di-Pietro.Ern:12}.

\section{Discretization}\label{sec:discretization}

In this section we define the discrete counterparts of the elasticity, hydro-mechanical coupling, and Darcy operators, and formulate the HHO--dG scheme for problem \eqref{eq:weak_form}.

\subsection{Nonlinear elasticity operator}

The discretization of the nonlinear elasticity operator closely follows~\cite{Botti.Di-Pietro.Sochala:17}.
We define the local symmetric gradient reconstruction $\GTs:\UT\to\Pmats[k]{T}$ such that, for a given $\uvv[T]=\big(\vv[T], (\vv[F])_{F\in\Fh[T]}\big)\in\UT$, $\GTs\uvv[T]\in\Pmats[k]{T}$ solves
\begin{equation}
  \label{eq:GTs}
  \int_T\GTs\uvv[T]:\matr{\tau}
  = -\int_T\vv[T]\cdot(\DIV\matr{\tau})
  + \sum_{F\in\Fh[T]}\int_F\vv[F]\cdot(\matr{\tau}\normal_{TF})\qquad
   \forall\matr{\tau}\in\Pmats[k]{T}.
\end{equation}
Existence and uniqueness of $\GTs\uvv[T]$ follow from the Riesz representation theorem in $\Pmats[k]{T}$ for the $L^2(T)^{d\times d}$-inner product.
This definition is motivated by the following property.
\begin{proposition}[Commuting property for the local symmetric gradient reconstruction]
  For all $\vv\in\Hvec{T}$, it holds that
  \begin{equation}\label{eq:GTs.commuting}
    \GTs\IvT\vv = \vlproj[T]{k}(\GRADs\vv).
  \end{equation}  
\end{proposition}
\begin{remark}[Approximation properties of the local symmetric gradient reconstruction]
  The commuting property \eqref{eq:GTs.commuting} combined with \eqref{eq:approx.lproj} shows that $\GTs\IvT\vv$ optimally approximates $\GRADs\vv$ in $\Pmats[k]{T}$.
\end{remark}
\begin{proof}
  For all $\matr{\tau}\in\Pmat{T}$, we can write
  $$
    \begin{aligned}
      \int_T\GTs\IvT\vv:\matr{\tau} &= \int_T\GTs\IvT\vv:\matr{\tau}_\symm
      \\
      &= -\int_T\vlproj[T]{k}\vv\cdot(\DIV\matr{\tau}_\symm)
      +\sum_{F\in\Fh[T]}\int_F\vlproj[F]{k}\vv\cdot(\matr{\tau}_\symm\normal_{TF}) 
      \\
      &= -\int_T\vv\cdot(\DIV\matr{\tau}_\symm) + \sum_{F\in\Fh[T]}\int_F\vv\cdot(\matr{\tau}_\symm\normal_{TF})
      \\
      &=\int_T \GRAD\vv:\matr{\tau}_\symm
      =\int_T \GRADs\vv:\matr{\tau}_\symm
      =\int_T \GRADs\vv:\matr{\tau}
      =\int_T\vlproj[T]{k}(\GRADs\vv):\matr{\tau},
    \end{aligned}
  $$
  where we have used \eqref{eq:magic} with $\matr{\alpha}=\GTs\IvT\vv$ and $\matr{\beta}=\matr{\tau}$ in the first line,
  the definition \eqref{eq:GTs} of the local symmetric gradient with $\uvv[T]=\IvT\vv$ in the second line, and
  definition \eqref{eq:lproj} after observing that $\DIV\matr{\tau}_\symm\in\Pvec[k-1]{T}\subset\Pvec{T}$ and 
  $\matr{\tau}_\symm\normal_{TF}\in\Pvec{F}$ for all $F\in\Fh[T]$ to remove the $L^2$-orthogonal projectors in the third 
  line. In the fourth line, we have used an integration by parts, then invoked \eqref{eq:magic} first with 
  $\matr{\alpha}=\matr{\tau}_\symm$ and $\matr{\beta}=\GRAD\vv$, then with $\matr{\alpha}=\GRADs\vv$ and 
  $\matr{\beta}=\matr{\tau}$, and we have used the definition \eqref{eq:lproj} of $\vlproj[T]{k}$ to conclude.
\end{proof}
From $\GTs$, we define the local displacement reconstruction operator 
$\rT:\UT\to\Pvec[k+1]{T}$ such that, for all $\uvv[T]\in\UT$,
\begin{align*}
  \int_T(\GRADs\rT\uvv[T]-\GTs\uvv[T]):\GRADs\vw &= 0 \qquad\forall\vw\in\Pvec[k+1]{T},
  \\
  \int_T\rT\uvv[T] &= \int_T\vv[T],\qquad
\end{align*}
and, denoting by $\partial_i$ the partial derivative with respect to the $i$th space variable, if $d=2$,
$$
\int_T\left(
\partial_1 r_{T,2}^{k+1}\uvv[T] - \partial_2 r_{T,1}^{k+1}\uvv[T]
\right)
=\sum_{F\in\Fh[T]}\int_F\left(
n_{TF,1} v_{F,2} - v_1 n_{TF,2}
\right),
$$
while, if $d=3$,
$$
\int_T\begin{pmatrix}
  \partial_2 r_{T,3}^{k+1}\uvv[T] - \partial_3 r_{T,2}^{k+1}\uvv[T] \\
  \partial_3 r_{T,1}^{k+1}\uvv[T] - \partial_1 r_{T,3}^{k+1}\uvv[T] \\
  \partial_1 r_{T,2}^{k+1}\uvv[T] - \partial_2 r_{T,1}^{k+1}\uvv[T]
\end{pmatrix}
= \sum_{F\in\Fh[T]}\int_F\begin{pmatrix}
n_{TF,2} v_{F,3} - n_{TF,3} v_{F,2} \\
n_{TF,3} v_{F,1} - n_{TF,1} v_{F,3} \\
n_{TF,1} v_{F,2} - n_{TF,2} v_{F,1}
\end{pmatrix}.
$$
Optimal approximation properties for $\rT\IvT$ have been recently proved in \cite[Appendix A]{Botti.Di-Pietro.ea:18} generalizing the ones of \cite[Lemma 2]{Di-Pietro.Ern:15}. The optimal approximation properties of $\rT\IvT$ are required to infer \eqref{eq:ah.consist} below.

The discretization of the nonlinear elasticity operator is realized by the function $a_h:\Uh\times\Uh\to\Real$ such 
that, for all $\uvw[h],\uvv[h]\in\Uh$,
\begin{equation}
  \label{eq:ah}  
  a_h(\uvw[h],\uvv[h])\coloneq\sum_{T\in\Th}\left(
  \int_T{\ms(\cdot,\GTs\uvu[T]):\GTs\uvv[T]} +
  \sum_{F\in\Fh[T]}\frac{\gamma}{h_F}\int_F\vec{\varDelta}_{TF}^k\uvu[T]\cdot\vec{\varDelta}_{TF}^k\uvv[T]\right),
\end{equation}
where $\gamma>0$ denotes a user-dependent parameter and we penalize in a least-square sense the face-based residual $\vec{\varDelta}_{TF}^k:\UT\to\Pvec[k]{F}$ such that, 
for all $T\in\Th$, all $\uvv[T]\in\UT$, and all $F\in\Fh[T]$, 
$$
\vec{\varDelta}_{TF}^k\uvv[T]\coloneq\vlproj[F]{k}(\rT\uvv[T]-\vv[F]) - \vlproj[T]{k}(\rT\uvv[T]-\vv[T]).
$$
This definition ensures that $\vec{\varDelta}_{TF}^k$ vanishes whenever its argument is of the form $\IvT\vw$ with 
$\vw\in\Pvec[k+1]{T}$, a crucial property to obtain high-order error estimates 
(cf.~\cite[Theorem 12]{Boffi.Botti.Di-Pietro:16}). 
For further use, we note the following seminorm equivalence, which can be proved using the arguments 
of~\cite[Lemma~4]{Di-Pietro.Ern:15}: For all $\uvv[h]\in\Uh$,
\begin{equation}
  \label{eq:norm_equivalence}
  C_{\rm{eq}}^{-2}\norm[\strain,h]{\uvv[h]}^2\le
  \sum_{T\in\Th}\left(\norm[T]{\GTs\uvv[T]}^2
  +\sum_{F\in\Fh[T]}h_F^{-1}\norm[F]{\vec{\varDelta}_{TF}^k\uvv[T]}^2\right)
  \le C_{\rm{eq}}^2\norm[\strain,h]{\uvv[h]}^2,
\end{equation}
where $C_{\rm{eq}}>0$ is independent of $h$,
and the discrete strain seminorm $\norm[\strain,h]{{\cdot}}$ is defined by \eqref{eq:strain.seminorm}.
By~\eqref{eq:hypo.coercivity}, this implies the coercivity of $a_h$. 

\begin{remark}[Choice of the stabilization parameter]\label{rem:gamma}
  The constants $C_{\rm gr},C_{\rm cv}$ appearing in~\eqref{eq:hypo} satisfy $C_{\rm cv}^2\le C_{\rm gr}$. Indeed, owing   
  to~\eqref{eq:hypo.coercivity}, the Cauchy--Schwarz inequality, and~\eqref{eq:hypo.growth}, it holds 
  for all $\matr{\tau}\in\Real^{d\times d}_\symm$,
  \begin{equation}
    \label{eq:const_rel1}
    C_{\rm cv}^2\seminorm[d\times d]{\matr{\tau}}^2
    \le \ms(\vec{x},\matr{\tau}):\matr{\tau}
    \le \seminorm[d\times d]{\ms(\vec{x},\matr{\tau})}\seminorm[d\times d]{\matr{\tau}}
    \le C_{\rm gr} \seminorm[d\times d]{\matr{\tau}}^2.
  \end{equation}
  Thus, we choose the stabilization parameter $\gamma$ in \eqref{eq:ah} such that 
  \begin{equation}
    \label{eq:penalty_mech}
    \gamma\in [C_{\rm cv}^2,C_{\rm gr}].
  \end{equation}
  For the linear elasticity model \eqref{eq:LinCauchy}, we have $C_{\rm gr}=2\mu+d\lambda$ and 
  $C_{\rm cv}=\sqrt{2\mu}$, so that a natural choice for the stabilization parameter is $\gamma = 2\mu$.
\end{remark} 

\subsection{Hydro-mechanical coupling}
The hydro-mechanical coupling is realized by means of the bilinear form $b_h$ on $\Uh\times P^{k}(\Th)$ such that, 
for all $\uvv[h]\in\Uh$ and all $q_h\in P^{k}(\Th)$,
\begin{equation}
  \label{eq:bh}
  b_h(\uvv[h],q_h)
  \coloneq
  \sum_{T\in\Th}\left(
  \int_T\vv[T]\cdot\GRAD q_T - \sum_{F\in\Fh[T]}\int_F(\vv[F]\cdot\normal_{TF})~q_T
  \right),
\end{equation}
where $q_T\coloneq q_{h|T}$ for all $T\in\Th$.
It can be checked using Cauchy--Schwarz inequalities together with the definition \eqref{eq:strain.seminorm} of the strain seminorm and discrete trace inequalities that there exists $C_{{\rm bd}}>0$ independent of $h$ such that 
$$
b_h(\uvv[h],q_h)\le C_{\rm{bd}}\norm[\strain,h]{\uvv[h]}\norm[\Omega]{q_h}.
$$
Additionally, using the strongly enforced boundary condition in $\UhD$, it can be proved that
\begin{equation}\label{eq:bh(vh,1)=0}
  b_{h}(\uvv[h],1)=0,\qquad\forall\uvv[h]\in\UhD.
\end{equation}
Finally, we note the following lemma stating that the hybrid interpolator $\Ivh:\Hvec{\Omega}\to\Uh$ is a Fortin operator.
\begin{lemma}[Fortin operator]
For all $\vv\in\Hvec{\Omega}$ and all $q_h\in P^{k}(\Th)$, the interpolator $\Ivh$ satisfies
\begin{subequations}
\begin{align}
  \label{eq:fortin.st}
  \norm[\strain,h]{\Ivh\vv}&\le C_{\rm st}\seminorm[1,\Omega]{\vv},
  \\
  \label{eq:fortin.com}
  b_h(\Ivh\vv,q_h) &= b(\vv,q_h),
\end{align}
\end{subequations}
where the strictly positive real number $C_{\rm st}$ is independent of $h$.
\end{lemma}
\begin{proof}
  (i) \emph{Proof of \eqref{eq:fortin.st}.}
  Recalling the definitions \eqref{eq:strain.seminorm} of the discrete strain seminorm and \eqref{eq:Ivh} of the global interpolator, we can write
  $$
  \begin{aligned}
    \norm[\strain,h]{\Ivh\vv}^2
    &= \sum_{T\in\Th}\left(
    \norm[T]{\GRADs\vlproj[T]{k}\vv}^2
    + \sum_{F\in\Fh[T]}h_F^{-1}\norm[F]{\vlproj[F]{k}\vv - \vlproj[T]{k}\vv}^2
    \right)
    \\
    &\le \sum_{T\in\Th}\left(
    2\norm[T]{\GRADs(\vlproj[T]{k}\vv-\vv)}^2
    + 2\norm[T]{\GRADs\vv}^2
    + \sum_{F\in\Fh[T]}h_F^{-1}\norm[F]{\vv - \vlproj[T]{k}\vv}^2
    \right)
    \le C_{\rm st}\seminorm[1,\Omega]{\vv}^2,
  \end{aligned}
  $$
  with $C_{\rm st}>0$ independent of $h$. To pass to the second line, we have used a triangle inequality after 
  inserting $\pm\GRADs\vv$ into the first term, and we have used the linearity, idempotency, and boundedness of 
  $\vlproj[F]{k}$ to write $\norm[F]{\vlproj[F]{k}\vv-\vlproj[T]{k}\vv}=\norm[F]{\vlproj[F]{k}(\vv-\vlproj[T]{k}\vv)}\le\norm[F]{\vv-\vlproj[T]{k}\vv}$.
  To conclude, we have used \eqref{eq:approx.lproj} with $l=m=1$ and \eqref{eq:approx.lproj:trace} with $l=1$ and $m=0$ to 
  bound the first and third term inside the summation.
  \medskip\\
  (ii) \emph{Proof of \eqref{eq:fortin.com}.}
  Recalling the definitions \eqref{eq:bh} of $b_h(\cdot,\cdot)$ and \eqref{eq:Ivh} of the global interpolator, 
  we can write letting, for the sake of brevity, $q_T\coloneq q_{h|T}$ for all $T\in\Th$,
  for all $q_h\in P^k(\Th)$,
  $$
  \begin{aligned}
    b_h(\Ivh\vv,q_h)
    &= \sum_{T\in\Th}\left(
    \int_T\vlproj[T]{k}\vv\cdot\GRAD q_T
    -\int_F(\vlproj[F]{k}\vv[|F]\cdot\normal_{TF})~q_T
    \right)
    \\
    &= \sum_{T\in\Th}\left(
    \int_T\vv\cdot\GRAD q_T
    -\int_F(\vv\cdot\normal_{TF})~q_T
    \right)
    = b(\vv,q_h),
  \end{aligned}
  $$
  where we have used definition \eqref{eq:lproj} after observing that $\GRAD q_T\in\Pvec[k-1]{T}\subset\Pvec[k]{T}$ and 
  $q_{T|F}\normal_{TF}\in\Pvec[k]{F}$ to remove the $L^2$-orthogonal projectors in the second line, and integration by 
  parts over $T\in\Th$ to conclude.
\end{proof}
As a result of the previous Lemma, one has the following inf-sup condition, cf. \cite{FVCA8} for the proof, which follows the classical Fortin argument (see, e.g., \cite[Section 8.4]{Boffi.Brezzi.ea:13} for further details).
\begin{proposition}
  There is a strictly positive real number $\beta$ independent of $h$ such that, for all $q_h\in P^{k}_0(\Th)$,
  \begin{equation}\label{eq:inf-sup}
    \norm[\Omega]{q_h}\le\beta
    \sup_{\uvv[h]\in\UhD\setminus\{\underline{\vec{0}}\}}\frac{b_h(\uvv[h],q_h)}{\norm[\strain,h]{\uvv[h]}}.
  \end{equation}
\end{proposition}

\subsection{Darcy operator}
The discretization of the Darcy operator is based on the Symmetric Weighted Interior Penalty method of~\cite{Di-Pietro.Ern.ea:08}, cf. also~\cite[Section~4.5]{Di-Pietro.Ern:12}.
For all $F\in\Fhi$ and all $q_h\in P^{k}(\Th)$, we define the jump and weighted average operators such that
$$
\jump{q_h}\coloneq q_{T_{F,1}}-q_{T_{F,2}},\qquad
\wavg{q_h}\coloneq
\frac{\sqrt{\sdiff[F,2]}}{\sqrt{\sdiff[F,1]}+\sqrt{\sdiff[F,2]}}q_{T_{F,1}} 
+ \frac{\sqrt{\sdiff[F,1]}}{\sqrt{\sdiff[F,1]}+\sqrt{\sdiff[F,2]}}q_{T_{F,2}},
$$
with $T_{F,1}, T_{F,2}\in\Th$, $T_{F,1}\neq T_{F,2}$, such that $F\subset \partial T_{F,1}\cap \partial T_{F,2}$ and 
$\sdiff[F,1],\sdiff[F,2]$ defined in \eqref{eq:diff_coeff}.
The bilinear form $c_h$ on $P^k(\Th)\times P^k(\Th)$ is defined such that, for all $q_h,r_h\in P^k(\Th)$,
$$
  \begin{aligned}
    c_h(r_h,q_h)
    \coloneq&
    \int_\Omega\diff\GRADh r_h\cdot\GRADh q_h + 
    \sum_{F\in\Fhi}\frac{\varsigma \sdiff}{h_F}
    \int_F\jump{r_h}\jump{q_h}
    \\
    &-\sum_{F\in\Fhi}\int_F\left(
    \jump{r_h}\wavg{\diff\GRADh q_h}
    +\jump{q_h}\wavg{\diff\GRADh r_h}
    \right)\cdot\normal_{T_1F},
  \end{aligned}
$$
where we have introduced the broken gradient operator $\GRADh$ on $\Th$
and we have denoted by $\varsigma>\underline{\varsigma}>0$ a user-defined penalty parameter chosen large enough to 
ensure the coercivity of $c_h$ (the proof is similar to~\cite[Lemma 4.51]{Di-Pietro.Ern:12}): 
$$
c_h(q_h,q_h)\ge(\varsigma-\underline{\varsigma})(1 + \varsigma)^{-1} \norm[\diff,h]{q_h}^2, \qquad \forall q_h\in\Ph.
$$
Since, under this condition, $c_h$ is a symmetric positive definite bilinear form on the broken polynomial space $\Ph$, we can define an associated norm by setting $\norm[c,h]{\cdot}\coloneq c_h(\cdot,\cdot)^{\frac12}$.

The following consistency result can be proved adapting the arguments of~\cite[Chapter~4]{Di-Pietro.Ern:12} to homogeneous Neumann boundary conditions and will be instrumental for the analysis.
We define the functional spaces $P_*\coloneq\left\{r\in H^1(\Omega)\cap H^2({P_\Omega})\st\text{$\diff\GRAD r\cdot\normal=0$ on $\partial\Omega$}\right\}$ and set $P_{*h}^k\coloneq P_* + \Ph$. Extending the bilinear form $c_h$ to 
$P_{*h}^k\times P_{*h}^k$, it is inferred that, for all $r\in P_*$, 
\begin{equation}
  \label{eq:ch.consist}
  -(\DIV(\diff\GRAD r), q)_\Omega
  = c_h(r, q)\qquad\forall q\in P_{*h}.
\end{equation}

\subsection{Discrete problem}\label{sec:disc.pb}

For all $1\le n\le N$, the discrete solution $(\uvu[h]^n,p_h^n)\in\UhD\times\Ph$ at time $t^n$ is such that, 
for all $(\uvv[h],q_{h})\in\UhD\times P^k(\Th)$,
\begin{subequations}\label{eq:nl_biot.h}
  \begin{align}
    \label{eq:nl_biot.h:mech}
    a_h(\uvu[h]^n,\uvv[h]) + b_h(\uvv[h],p_h^n) &= (\ovf^n,\vv[h])_{\Omega},
    \\
    \label{eq:nl_biot.h:flow}
    C_0(\ddt p_h^n, q_h)_{\Omega} -b_h(\ddt\uvu[h]^n,q_h) + c_h(p_h^n,q_h) &= (\overline{g}^n,q_h)_{\Omega},
  \end{align}
with $\ovf^n\in\Lvec{\Omega}$ and $\overline{g}^n\in L^2(\Omega)$ defined according to \eqref{eq:time.average}. In order to start the time-stepping scheme, we need to initialize the discrete fluid content. This is done by setting 
$\phi_h^0$ equal to the $L^2$-orthogonal projection of $\phi^0$ on $P^k(\Th)$ according to \eqref{eq:weak_form.initial}, 
that is,
\begin{equation}
  \label{eq:nl_biot.h:initial}
  C_0(p_h^0, q_h)_{\Omega} -b_h(\uvu[h]^0,q_h) \coloneq (\phi^0, q_h)_{\Omega} \qquad\forall q_h\in P^k(\Th).
\end{equation}
\end{subequations}
\begin{remark}[Initial condition]
We observe that the initial displacement $\uvu[h]^0\in\UhD$ and pressure $p_h^0\in\Ph$ in \eqref{eq:nl_biot.h:initial} are not explicitly required to initialize the scheme. However, assuming that $\vf\in C^0(\Lvec{\Omega})$, so that~\eqref{eq:nl_biot.strong.mech} makes sense also for $t=0$, it is possible to compute the initial discrete fields $(\uvu[h]^0,p_h^0)$  by solving
\begin{subequations}\label{eq:initial.h}
  \begin{align}
    \label{eq:initial:mech}
    a_h(\uvu[h]^0,\uvv[h]) + b_h(\uvv[h],p_h^0) &= (\vf^0,\vv[h])_{\Omega} \qquad\forall \uvv[h]\in\UhD,
    \\
    \label{eq:initial:fluid}
    C_0(p_h^0, q_h)_{\Omega} -b_h(\uvu[h]^0,q_h) &= (\phi^0, q_h)_{\Omega} \ \qquad\forall q_h\in P^k(\Th).
  \end{align}
\end{subequations}
In the limit case $C_0=0$ the previous equations corresponds to a well-posed HHO discretization of a steady nonlinear Stokes-like problem. If $C_0>0$ we can take $q_h$ in \eqref{eq:initial:fluid} such that, for all $T\in\Th$, $(q_h)_{|T}=C_0^{-1}\optr(\GTs\uvv[T])$ and sum the resulting equation to \eqref{eq:initial:mech}. Owing to definitions \eqref{eq:GTs} and \eqref{eq:bh}, we obtain
\begin{equation}\label{eq:initial.elh}
  \tilde{a}_h(\uvu[h]^0,\uvv[h])= (\vf^0,\vv[h])_{\Omega} - C_0^{-1} b_h(\uvv[h],\lproj[h]{k}\phi^0),
\end{equation}
where the nonlinear function $\tilde{a}_h$ is defined as $a_h$ in \eqref{eq:ah} but replacing the stress-strain law $\ms$ with
$$
\tilde{\ms}(\cdot,\vec{\tau})\coloneq\ms(\cdot,\vec{\tau}) + C_0^{-1}\optr(\vec\tau)\Id.
$$ 
According to \cite[Theorem 7]{Botti.Di-Pietro.Sochala:17}, the nonlinear elasticity problem \eqref{eq:initial.elh} admits a solution. Once the initial displacement $\uvu[h]^0$ is computed, we set $p_h^0$ such that, for all $T\in\Th$, $(p_h^0)_{|T}=C_0^{-1} (\lproj[T]{k}\phi^0-\optr(\GTs\uvu[T]^0))$.
\end{remark}
\begin{remark}[Time discretization]
The modified backward Euler scheme obtained by taking time averages instead of pointwise evaluation of the right-hand sides in \eqref{eq:nl_biot.h} can be interpreted as a low-order discontinuous Galerkin time-stepping method, 
cf. \cite{Schautzau.Schwab:00,Smears:17}.
\end{remark}
Notice that other time discretizations could be used, but we have decided to focus on the backward Euler scheme to keep the proofs as simple as possible.
From the practical point of view, at each time step $n$, the discrete nonlinear system~\eqref{eq:nl_biot.h} can be solved by the Newton method using as initial guess the solution at step $(n-1)$.
The size of the linear system to be solved at each Newton iteration can be reduced by statically condensing a large part of the unknowns as described in \cite[Section 5]{Boffi.Botti.Di-Pietro:16}.


\section{Stability and well-posedness}\label{sec:stability}

In this section we study the stability of problem~\eqref{eq:nl_biot.h} and prove its well-posedness. We start with an a priori estimate on the discrete solution not requiring conditions on the time step $\tau$ and robust with respect to vanishing storage coefficients and small permeability.

\begin{proposition}[A priori estimate] \label{pro:a-priori}
  Denote by $(\uvu[h]^n,p_h^n)_{1\le n\le N}$ the solution to~\eqref{eq:nl_biot.h}. Under Assumption \ref{ass:hypo} on the stress-strain relation and the regularity on the data $\vf$, $g$, and $\phi^0$ assumed in Section \ref{sec:weak_form}, it holds
  \begin{multline}\label{eq:a-priori}
    \sum_{n=1}^N\tau\norm[\strain,h]{\uvu[h]^n}^2
    +\sum_{n=1}^N\tau\left(\norm[\Omega]{p_h^n-\lproj[\Omega]{0}p_h^n}^2 + C_0\norm[\Omega]{p_h^n}^2\right)
    + \norm[c,h]{s_h^N}^2
    \le
    \\ \quad C\left( 
    \norm[L^2(\Lvec{\Omega})]{\vf}^2+\tF^2\norm[L^2(L^2(\Omega))]{g}^2+\tF\norm[\Omega]{\phi^0}^2
    +\tF^2{C_0}^{-1}\norm[L^2(L^2(\Omega))]{\lproj[\Omega]{0}g}^2
    +\tF{C_0}^{-1}\norm[\Omega]{\lproj[\Omega]{0}\phi^0}^2.
    \right).
  \end{multline}
  where $C>0$ denotes a real number independent of $h$, $\tau$, the physical parameters $C_0$ 
  and $\diff$, and the final time $\tF$.
  In \eqref{eq:a-priori}, we have defined $s_h^N\coloneq\sum_{n=1}^N \tau p_h^n$ and we have adopted the 
  convention that $C_0^{-1}\norm[L^2(L^2(\Omega))]{\lproj[\Omega]{0}{g}}^2=0$ 
  and $C_0^{-1}\norm[\Omega]{\lproj[\Omega]{0}{\phi^0}}^2=0$ if $C_0=0$.
\end{proposition}
\begin{remark}
\label{rem:time_reg1}
In order to prove the a priori bound \eqref{eq:a-priori}, no additional time regularity assumption on the loading term 
$\vf$ and the mass source $g$ are needed, whereas the stability estimate of \cite[Lemma 7]{Boffi.Botti.Di-Pietro:16}, valid for linear stress-strain relation, requires $\vf\in C^1(\Lvec{\Omega})$ and $g\in C^0(L^2(\Omega))$. On the other hand, Proposition \ref{pro:a-priori} gives an estimate of the discrete displacement and pressure in the $L^2$-norm in time, while \cite[Lemma 7]{Boffi.Botti.Di-Pietro:16} ensures a control in the $L^\infty$-norm in time.
However, under additional requirements on the stress-strain law (for instance Assumption \ref{ass:hypo_add}) and 
$H^1$-regularity in time of $\vf$, a stronger version of \eqref{eq:a-priori} can be inferred, including in particular an estimate in the $L^\infty$-norm in time.
\end{remark}
\begin{proof}
  (i) \emph{Estimate of $\norm[\Omega]{p_h^n-\lproj[\Omega]{0}p_h^n}$.} 
  The growth property of the stress-strain function~\eqref{eq:hypo.growth} together with the 
  Cauchy--Schwarz inequality, assumption \eqref{eq:penalty_mech} on the stabilization parameter, and the second inequality 
  in~\eqref{eq:norm_equivalence} yield, for all $1\le n\le N$ and all $\uvv\in\UhD$,  
  \begin{equation}
  \begin{aligned}
    \label{eq:gr_bd}
    a_h(\uvu[h]^n,\uvv[h])
    &=
    \sum_{T\in\Th}\left(\int_T{\ms(\cdot,\GTs\uvu[T]^n):\GTs\uvv[T]} 
    +\sum_{F\in\Fh[T]}\frac{\gamma}{h_F}\int_F\vec{\varDelta}_{TF}^k\uvu[T]^n\cdot\vec{\varDelta}_{TF}^k\uvv[T]\right)
    \\ &\le C_{\rm gr}
    \sum_{T\in\Th}\left(\norm[T]{\GTs\uvu[T]^n}\norm[T]{\GTs\uvv[T]}
    +\sum_{F\in\Fh[T]}\frac{1}{h_F}\norm[F]{\vec{\varDelta}_{TF}^k\uvu[T]^n}\norm[F]{\vec{\varDelta}_{TF}^k\uvv[T]}\right)
    \\ &\le C_{\rm gr}C_{\rm eq}^2 
    \norm[\strain,h]{\uvu[h]^n}\norm[\strain,h]{\uvv[h]}.
  \end{aligned}
  \end{equation}
  Using the inf-sup condition~\eqref{eq:inf-sup}, \eqref{eq:bh(vh,1)=0}, and the 
  mechanical equilibrium equation~\eqref{eq:nl_biot.h:mech}, we get, for any $1\le n\le N$,
  $$
    \norm[\Omega]{p_h^n-\lproj[\Omega]{0}p_h^n} \le \beta
    \sup_{\uvv[h]\in\UhD\setminus\{\underline{\vec{0}}\}}
    \frac{b_h(\uvv[h], p_h^n-\lproj[\Omega]{0}p_h^n)}{\norm[\strain,h]{\uvv[h]}}
    = \beta\sup_{\uvv[h]\in\UhD\setminus\{\underline{\vec{0}}\}}
    \frac{(\ovf^n,\vv[h])_\Omega - a_h(\uvu[h]^n,\uvv[h])}{\norm[\strain,h]{\uvv[h]}}.
  $$
  Therefore, owing to the discrete Korn inequality \eqref{eq:Korn} and to \eqref{eq:gr_bd}, we infer from the previous 
  bound that 
  \begin{equation}
    \label{eq:stability:bnd.phn}
    \norm[\Omega]{p_h^n-\lproj[\Omega]{0}p_h^n}
    \le \beta\left(C_{\rm K}\norm[\Omega]{\ovf^n}+C_{\rm gr}C_{\rm eq}^2\norm[\strain,h]{\uvu[h]^n}\right).
  \end{equation}
  \\
  (ii) \emph{Energy balance.}
  For all $1\le n \le N$, summing~\eqref{eq:nl_biot.h:flow} at times $1\le i\le n$, taking $q_h=\tau^2 p_h^n$ as a test function, and recalling the discrete initial condition \eqref{eq:nl_biot.h:initial} yields
  \begin{equation}
    \label{eq:stab.fluid:1}
    \tau C_0(p_h^n, p_h^n)_\Omega - \tau b_h(\uvu[h]^n, p_h^n)+ \sum_{i=1}^n \tau^2 c_h(p_h^i, p_h^n)
    = \sum_{i=1}^n \tau^2 (\overline{g}^i, p_h^n)_\Omega + \tau(\phi^0, p_h^n)_\Omega.
  \end{equation} 
  Moreover, using the linearity of $c_h$ and the formula $2x(x-y) = x^2+(x-y)^2-y^2$, the third term in the left-hand 
  side of \eqref{eq:stab.fluid:1} can be rewritten as
  $$
    \sum_{i=1}^n \tau^2 c_h(p_h^i, p_h^n)
    = \tau c_h \left(\sum_{i=1}^n \tau p_h^i, p_h^n\right)
    = \tau c_h (s_h^n, \ddt s_h^n)
    =\frac12 \left(\norm[c,h]{s_h^n}^2 + \norm[c,h]{\ddt s_h^n}^2- \norm[c,h]{s_h^{n-1}}^2\right),
  $$
  where we have set $s_h^0\coloneq0$, $s_h^n\coloneq\sum_{i=1}^n\tau p_h^i$ for any $1\le n\le N$, and observed 
  that $p_h^n=\ddt s_h^n$. Therefore, summing~\eqref{eq:stab.fluid:1} and~\eqref{eq:nl_biot.h:mech} at discrete time $n$ 
  with $\uvv[h]=\tau\uvu[h]^n$, leads to
  $$
    \tau a_h(\uvu[h]^n, \uvu[h]^n) + \tau C_0 \norm[\Omega]{p_h^n}^2 + 
    \frac12 \left(\norm[c,h]{s_h^n}^2 - \norm[c,h]{s_h^{n-1}}^2\right)
    \le \tau(\ovf^n, \vu[h]^n)_\Omega + \sum_{i=1}^n \tau^2 (\overline{g}^i, p_h^n)_\Omega 
    + \tau(\phi^0, p_h^n)_\Omega.
  $$
  Summing the previous relation for $1\le n\le N$, telescoping out the appropriate summands, and using the coercivity 
  property~\eqref{eq:hypo.coercivity}, assumption \eqref{eq:penalty_mech}, and the first inequality 
  in~\eqref{eq:norm_equivalence}, we get
  \begin{equation}
    \label{eq:en.balance}
    \frac{C_{\rm cv}^2}{C_{\rm eq}^2}\sum_{n=1}^N \tau \norm[\strain,h]{\uvu[h]^n}^2
    + C_0 \sum_{n=1}^N \tau\norm[\Omega]{p_h^n}^2 + \frac12\norm[c,h]{s_h^N}^2
    \le \sum_{n=1}^N \tau (\ovf^n, \vu[h]^n)_\Omega 
    + \sum_{n=1}^N \tau (G^n + \phi^0, p_h^n)_\Omega,
  \end{equation}
  with the notation $G^n\coloneq \sum_{i=1}^n \tau\overline{g}^i=\int_0^{t^n}g(t) {\rm d}t$. 
  We denote by $\mathcal{R}$ the right-hand side of~\eqref{eq:en.balance} and proceed to find a suitable upper bound.
  \medskip\\
  (iii) \emph{Upper bound for $\mathcal{R}$.}
  For the first term in the right-hand side of~\eqref{eq:en.balance}, using the Cauchy--Schwarz, discrete 
  Korn \eqref{eq:Korn}, and Young inequalities, we obtain
  \begin{equation}
    \label{eq:stability:R1}
    \sum_{n=1}^N \tau (\ovf^n,\vu[h]^n)_\Omega
    \le C_{\rm K}\left(\sum_{n=1}^N \tau\norm[\Omega]{\ovf^n}^2\right)^{\frac12}
    \left(\sum_{n=1}^{N}\tau\norm[\strain,h]{\uvu[h]^n}^2\right)^{\frac12}
    \le \frac{C_{\rm K}^2 C_{\rm eq}^2}{C_{\rm cv}^2} \norm[L^2(\Lvec{\Omega})]{\vf}^2
    + \frac{C_{\rm cv}^2}{4 C_{\rm eq}^2}\sum_{n=1}^N\tau\norm[\strain,h]{\uvu[h]^n}^2,
  \end{equation}
  where we have used the Jensen inequality to infer that 
  $$
  \sum_{n=1}^N\tau\norm[\Omega]{\ovf^n}^2 
  =\sum_{n=1}^N\frac1\tau\int_\Omega\left(\int_{t^{n-1}}^{t^n}\vf(\vec{x},t) {\rm d}t\right)^2 {\rm d}\vec{x}
  \le\sum_{n=1}^N\int_{t^{n-1}}^{t^n}\norm[\Omega]{\vf(t)}^2{\rm d}t=\norm[L^2(\Lvec{\Omega})]{\vf}^2.
  $$
  We estimate the second term in the right-hand side of~\eqref{eq:en.balance} by splitting it into two 
  contributions as follows:
  $$
    \sum_{n=1}^N \tau (G^n+\phi^0, p_h^n)_\Omega = 
    \sum_{n=1}^N \tau (G^n+\phi^0,p_h^n-\lproj[\Omega]{0}p_h^n)_\Omega
    + \sum_{n=1}^N \tau (\lproj[\Omega]{0}(G^n+\phi^0), p_h^n)_\Omega \coloneq \term_1 + \term_2,
    $$
  where we have used its definition \eqref{eq:lproj} to move $\lproj[\Omega]{0}$ from $p_h^n$ to $(G^n+\phi^0)$ 
  in the second term. Owing to the Cauchy--Schwarz, triangle, Jensen, and Young inequalities, 
  and using~\eqref{eq:stability:bnd.phn}, we have 
  \begin{equation}
    \label{eq:stability:bnd.R2}
    \begin{aligned}
      \left|\term_1\right|
      &\le \left(\sum_{n=1}^N\tau\norm[\Omega]{G^n+\phi^0}^2\right)^{\frac12}
      \left(\sum_{n=1}^N\tau\norm[\Omega]{p_h^n-\lproj[\Omega]{0}p_h^n}^2\right)^{\frac12}
      \\
      &\le\sqrt2\beta\left(
      \sum_{n=1}^N\tau\int_\Omega\left(\int_0^{t^n}\hspace{-2mm}g(\vec{x},t) {\rm d}t\right)^2\hspace{-1mm}{\rm d}\vec{x}
      +\tF\norm[\Omega]{\phi^0}^2\right)^{\frac12}
      \left(\sum_{n=1}^N\tau
      \left(C_{\rm K}\norm[\Omega]{\ovf^n}+C_{\rm gr}C_{\rm eq}^2\norm[\strain,h]{\uvu[h]^n}\right)^2
      \right)^{\frac12}
      \\
      &\le 
      2\beta\left(
      \tF\norm[L^2(L^2(\Omega))]{g}+\tF^{\frac12}\norm[\Omega]{\phi^0}
      \right)
      \left[C_{\rm K}\norm[L^2(\Lvec{\Omega})]{\vf}+
      C_{\rm gr}C_{\rm eq}^2\left(\sum_{n=1}^N\tau\norm[\strain,h]{\uvu[h]^n}^2\right)^{\frac12}\right]
      \\
      &\le
      \tF\beta^2\left(1+\frac{C_{\rm gr}^2C_{\rm eq}^6}{C_{\rm cv}^2}\right)
      \left(\tF^{\frac12}\norm[L^2(L^2(\Omega))]{g}+\norm[\Omega]{\phi^0}\right)^2
      + C_{\rm K}^2 \norm[L^2(\Lvec{\Omega})]{\vf}^2
      + \frac{C_{\rm cv}^2}{4 C_{\rm eq}^2}\sum_{n=1}^N\tau\norm[\strain,h]{\uvu[h]^n}^2.
    \end{aligned}
  \end{equation}
  Owing to the compatibility condition~\eqref{eq:compatibility} and the linearity of the $L^2$-projector, $\term_2=0$ 
  if $C_0=0$. Otherwise, using again the Cauchy--Schwarz, triangle, Jensen, and Young inequalities, leads to
  \begin{equation}  
    \label{eq:stability:bnd.R2b}
      \begin{aligned}
      \left|\term_2\right|
      &\le \left(\sum_{n=1}^N\tau\norm[\Omega]{\lproj[\Omega]{0}G^n+\lproj[\Omega]{0}\phi^0}^2\right)^{\frac12}      
      \left(\sum_{n=1}^N\tau \norm[\Omega]{p_h^n}^2\right)^{\frac12}
      \\
      &\le
      \sqrt2\left(\tF^2\norm[L^2(L^2(\Omega))]{\lproj[\Omega]{0}g}^2+\tF\norm[\Omega]{\lproj[\Omega]{0}\phi^0}^2
      \right)^{\frac12}
      \left(\sum_{n=1}^N\tau \norm[\Omega]{p_h^n}^2\right)^{\frac12}
      \\
      &\le
      \frac{2\tF^2}{3C_0}\norm[L^2(L^2(\Omega))]{\lproj[\Omega]{0}g}^2
      +\frac{2\tF}{3C_0}\norm[\Omega]{\lproj[\Omega]{0}\phi^0}^2
      +\frac{3C_0}4\sum_{n=1}^N\tau\norm[\Omega]{p_h^n}^2.
      \end{aligned}
  \end{equation}
  Finally, from~\eqref{eq:stability:R1},~\eqref{eq:stability:bnd.R2},~\eqref{eq:stability:bnd.R2b}, it follows that
  \begin{multline}
    \label{eq:stability:bnd.R}
    \mathcal{R}
    \le
    \frac{C_{\rm cv}^2}{2 C_{\rm eq}^2}\sum_{n=1}^N\tau\norm[\strain,h]{\uvu[h]^n}^2
    +\frac{3C_0}4\sum_{n=1}^N\tau\norm[\Omega]{p_h^n}^2
    +\frac{2\tF^2}{3C_0}\norm[L^2(L^2(\Omega))]{\lproj[\Omega]{0}g}^2
    +\frac{2\tF}{3C_0}\norm[\Omega]{\lproj[\Omega]{0}\phi^0}^2
    \\
    +C_{\rm K}^2 C_{\rm cv}^{-2}\left(C_{\rm eq}^2+C_{\rm cv}^2\right)\norm[L^2(\Lvec{\Omega})]{\vf}^2
    +\tF\beta^2C_{\rm cv}^{-2} \left(4 C_{\rm gr}^2 C_{\rm eq}^6+C_{\rm cv}^{2}\right)
    \left(\tF^{\frac12}\norm[L^2(L^2(\Omega))]{g}+\norm[\Omega]{\phi^0}\right)^2.
  \end{multline}
  \\
  (iv) \emph{Conclusion.}
  Passing the first two terms in the right-hand side of~\eqref{eq:stability:bnd.R} to the left-hand side 
  of~\eqref{eq:en.balance} and multiplying both sides by a factor $4$, we obtain
  \begin{equation}
    \label{eq:non.robust.bnd}
    \frac{2 C_{\rm cv}^2}{C_{\rm eq}^2}\sum_{n=1}^N\tau\norm[\strain,h]{\uvu[h]^n}^2
    + C_0\sum_{n=1}^N\tau\norm[\Omega]{p_h^n}^2
    + 2\norm[c,h]{s_h^N}^2
    \le 4\mathcal{C},
  \end{equation}
  where we have denoted by $\mathcal{C}$ the last four summands in the right-hand side of~\eqref{eq:stability:bnd.R}.
  In order to conclude we apply again~\eqref{eq:stability:bnd.phn} to obtain a bound of the $L^2$-norm of the 
  discrete pressure independent of the storage coefficient $C_0$. Indeed, owing to~\eqref{eq:stability:bnd.phn} 
  and $C_{\rm cv}^2\le C_{\rm gr}$ (see Remark \ref{rem:gamma}), it is inferred that
  $$
    \frac{C_{\rm cv}^2}{2C_{\rm gr}^2 C_{\rm eq}^6}\sum_{n=1}^N\tau\norm[\Omega]{p_h^n-\lproj[\Omega]{0} p_h^n}^2
    \le 
    \frac{C_{\rm cv}^2}{C_{\rm eq}^2}\sum_{n=1}^N\tau\norm[\strain,h]{\uvu[h]^n}^2
    +\frac{\beta^2 C_{\rm K}^2}{C_{\rm cv}^2 C_{\rm eq}^6}\norm[L^2(\Lvec{\Omega})]{\vf}^2.
  $$
  Summing the previous relation to \eqref{eq:non.robust.bnd} yields
  \begin{multline*}
    \frac{C_{\rm cv}^2}{C_{\rm eq}^2}\sum_{n=1}^N\tau\norm[\strain,h]{\uvu[h]^n}^2
    +\frac{C_{\rm cv}^2}{2C_{\rm gr}^2 C_{\rm eq}^6}\sum_{n=1}^N\tau\norm[\Omega]{p_h^n-\lproj[\Omega]{0}p_h^n}^2
    + C_0\sum_{n=1}^N\tau\norm[\Omega]{p_h^n}^2
    + 2\norm[c,h]{s_h^N}^2
    \le
    \\
    C_{\rm K}^2 \left(\frac{4 C_{\rm eq}^2}{C_{\rm cv}^2}+\frac{\beta^2}{C_{\rm eq}^6 C_{\rm cv}^2}+4\right)
    \norm[L^2(\Lvec{\Omega})]{\vf}^2
    +4\tF\beta^2 \left(\frac{4 C_{\rm gr}^2 C_{\rm eq}^6}{C_{\rm cv}^2}+1\right)
    (\tF^{\frac12}\norm[L^2(L^2(\Omega))]{g}+\norm[\Omega]{\phi^0})^2
    \\
    +\frac{8\tF^2}{3C_0}\norm[L^2(L^2(\Omega))]{\lproj[\Omega]{0}g}^2
    +\frac{8\tF}{3C_0}\norm[\Omega]{\lproj[\Omega]{0}\phi^0}^2.
  \end{multline*}
  Thus, multiplying both sides of the previous relation by $\max\{ C_{\rm eq}^2 C_{\rm cv}^{-2},\, 
  2C_{\rm gr}^2 C_{\rm eq}^6 C_{\rm cv}^{-2},\, 1\}$ gives \eqref{eq:a-priori}.
\end{proof}
\begin{remark}[A priori bound for $C_0=0$]
  When $C_0=0$, the a priori bound \eqref{eq:a-priori} reads
  $$
    \sum_{n=1}^N\tau\norm[\strain,h]{\uvu[h]^n}^2
    +\sum_{n=1}^N\tau\norm[\Omega]{p_h^n}^2
    +\norm[c,h]{s_h^N}^2
    \le C\left(\norm[L^2(\Lvec{\Omega})]{\vf}^2+\tF^2\norm[L^2(L^2(\Omega))]{g}^2+\tF\norm[\Omega]{\phi^0}^2\right).
  $$
  The conventions $C_0^{-1}\norm[L^2(L^2(\Omega))]{\lproj[\Omega]{0}g}^2=0$ and 
  $C_0^{-1}\norm[\Omega]{\lproj[\Omega]{0}{\phi^0}}^2=0$ if $C_0=0$ are justified since the term 
  $\term_{2}$ in point (3) of the previous proof vanishes in this case thanks to the compatibility 
  condition~\eqref{eq:compatibility}.
\end{remark}

We next proceed to discuss the existence and uniqueness of the discrete solutions. The proof of the following theorem hinges on the arguments of~\cite[Theorem 3.3]{Deimling:85}.
\begin{theorem}[Existence and uniqueness]
  Let Assumption~\ref{ass:hypo} hold and let $(\Mh)_{h\in{\cal H}}$ be a regular mesh sequence. Then, for all 
  $h\in{\cal H}$ and all $N\in\Natural^*$, there exists a unique solution 
  $(\uvu[h]^n,p_h^n)_{1\le n \le N}\in(\UhD\times\Ph)^N$ to~\eqref{eq:nl_biot.h}.
\end{theorem}
\begin{proof}
  We define the linear stress-strain function ${\ms}^{\rm lin}:\Real^{d\times d}_\symm\to\Real^{d\times d}_\symm$ such that, for all $\matr{\tau}\in\Real^{d\times d}_\symm$, 
  $$
  {\ms}^{\rm lin}(\matr{\tau})=\frac{C_{\rm cv}^2}{2}\matr{\tau} + \frac{C_{\rm cv}^2}{2d}\optr(\matr{\tau})\Id,
  $$
  where $C_{\rm cv}$ is the coercivity constant of $\matr{\sigma}$ (see \eqref{eq:hypo.coercivity}), and we denote by 
  $a_h^{\rm lin}$ the bilinear form obtained by replacing $\ms$ with ${\ms}^{\rm lin}$ in \eqref{eq:ah}.
  We consider the following auxiliary linear problem: For all $1\le n\le N$, find $(\uvy[h]^n, p_h^n)\in\UhD\times\Ph$ 
  such that 
  \begin{equation}\label{eq:nl_biot.h:lin}
    \begin{aligned}
      a_h^{\rm lin}(\uvy[h]^n,\uvv[h])+b_h(\uvv[h],p_h^n) &= (\ovf^n,\vv[h])_{\Omega} \qquad\forall\uvv[h]\in\UhD,
      \\
      C_0(\ddt p_h^n, q_h)_{\Omega} -b_h(\ddt\uvy[h]^n,q_h) + c_h(p_h^n,q_h) &= (\overline{g}^n,q_h)_{\Omega}
      \qquad\forall q_h\in\Ph,
    \end{aligned}
  \end{equation}
  with initial condition as in \eqref{eq:nl_biot.h:initial}.
  Since the previous system is linear and square and its solution satisfies the a priori estimate of 
  Proposition \ref{pro:a-priori}, it is readily inferred that problem \eqref{eq:nl_biot.h:lin} admits a unique solution.
  
  Now we observe that, thanks to the norm equivalence~\eqref{eq:norm_equivalence}, $a_h^{\rm lin}(\cdot,\cdot)$ is 
  a scalar product on $\UhD$, and we define the mapping $\underline{\vec{\Phi}}_h:\UhD\to\UhD$ such that, 
  for all $\uvv[h]\in\UhD$, 
  $$
  a_h^{\rm lin}(\underline{\vec{\Phi}}_h(\uvv[h]),\uvw[h])=a_h(\uvv[h],\uvw[h]),\qquad\forall\uvw[h]\in\UhD.
  $$
  We want to show that $\underline{\vec{\Phi}}_h$ is an isomorphism.
  Let $\uvv[h],\uvz[h]\in\UhD$ be such that $\underline{\vec{\Phi}}_h(\uvv[h])=\underline{\vec{\Phi}}_h(\uvz[h])$. If $\uvv[h]\neq\uvz[h]$, owing to the norm 
  equivalence~\eqref{eq:norm_equivalence} and the fact that $\norm[\strain,h]{\cdot}$ is a norm on $\UhD$, there is at 
  least one $T\in\Th$ such that 
  $\GTs\uvv[T]\neq\GTs\uvz[T]$ or $\vec{\varDelta}_{TF}^k\uvv[T]\neq\vec{\varDelta}_{TF}^k\uvz[T]$ for some $F\in\Fh[T]$.
  In both cases, owing to the definition of $a_h$ and 
  the strict monotonicity assumption~\eqref{eq:hypo.monotonicity}, it holds
  $$
  0 < a_h(\uvv[h],\uvv[h]-\uvz[h])-a_h(\uvz[h],\uvv[h]-\uvz[h])
  = a_h^{\rm lin}(\underline{\vec{\Phi}}_h(\uvv[h])-\underline{\vec{\Phi}}_h(\uvz[h]),\uvv[h]-\uvz[h])=0.
  $$
  Thus, we infer by contradiction that $\uvv[h]=\uvz[h]$ and, as a result, $\underline{\vec{\Phi}}_h$ is injective.
  In order to prove that $\underline{\vec{\Phi}}_h$ is also onto, we recall the following result: If $(E,(\cdot,\cdot)_E)$ is a Euclidean 
  space and $\Psi:E\to E$ is a continuous map such that $\frac{(\Psi(x),x)_E}{\norm[E]{x}}\to+\infty$ 
  as $\norm[E]{x}\to+\infty$, then $\Psi$ is surjective.  
  Since $(\UhD,\,a_h^{\rm lin}(\cdot,\cdot))$ is a Euclidean space and the coercivity~\eqref{eq:hypo.coercivity} of 
  $\ms$ together with the definition of $\ms^{\rm lin}$ yield
  $a_h^{\rm lin}(\underline{\vec{\Phi}}_h(\uvv[h]),\uvv[h])\ge a_h^{\rm lin}(\uvv[h],\uvv[h])$ for all $\uvv[h]\in\UhD$, 
  we deduce that $\underline{\vec{\Phi}}_h$ is an isomorphism. 
  Let, for all $1\le n \le N$, $(\uvy[h]^n, p_h^n)\in\UhD\times\Ph$ be the solution to problem \eqref{eq:nl_biot.h:lin}. 
  By the surjectivity and injectivity of $\underline{\vec{\Phi}}_h$, for all $1\le n \le N$, there exists a unique 
  $\uvu[h]^n\in\UhD$ such that $\underline{\vec{\Phi}}_h(\uvu[h]^n)=\uvy[h]^n$. 
  By definition of $\underline{\vec{\Phi}}_h$ and $(\uvy[h]^n)_{1\le n \le N}$, $(\uvu[h]^n, p_h^n)_{1\le n \le N}$ is 
  therefore the unique solution of the discrete problem \eqref{eq:nl_biot.h}.
\end{proof}


\section{Convergence analysis}\label{sec:convergence}

In this section we study the convergence of problem~\eqref{eq:nl_biot.h} and prove optimal error estimates under the following additional assumptions on the stress-strain function $\ms$.
\begin{assumption}[Stress-strain relation II]
  \label{ass:hypo_add}
There exist real numbers $C_{\rm lp},C_{\rm mn}\in(0,+\infty)$ such that, for a.e. $\vec{x}\in\Omega$, and all 
$\matr{\tau},\matr{\eta}\in\Real^{d\times d}_\symm$,
\begin{subequations}
  \label{eq:hypo_add}
  \begin{alignat}{2} 
    \label{eq:hypo_add.lip}
    &\seminorm[d\times d]{\ms(\vec{x},\matr{\tau})-\ms(\vec x,\matr{\eta})}\le C_{\rm lp}
    \seminorm[d\times d]{\matr{\tau}-\matr{\eta}},
    &\quad &\text{(Lipschitz continuity)}
    \\ 
    \label{eq:hypo_add.smon}
    &\left(\ms(\vec{x},\matr{\tau})-\ms(\vec{x},\matr{\eta})\right):\left(\matr{\tau}-\matr{\eta}\right)\ge C_{\rm mn}^2
    \seminorm[d\times d]{\matr{\tau}-\vec\eta}^2.
    &\quad &\text{(strong monotonicity)}
  \end{alignat}
\end{subequations}
\end{assumption}
\begin{remark}[Lipschitz continuity and strong monotonocity]
It is readily seen, by taking $\matr{\eta}=\matr{0}$ in \eqref{eq:hypo_add}, that Lipschitz continuity and strong monotonicity imply respectively the growth and coercivity properties of Assumption \ref{ass:hypo}. Therefore, recalling \eqref{eq:const_rel1}, it is inferred that the constants appearing in \eqref{eq:hypo.growth}, \eqref{eq:hypo.coercivity}, \eqref{eq:hypo_add.lip}, and \eqref{eq:hypo_add.smon} satisfy
\begin{equation}
  \label{eq:const_rel2}
  C_{\rm mn}^2\le C_{\rm cv}^2 \le C_{\rm gr} \le C_{\rm lp}.
\end{equation}
It was proved in~\cite[Lemma 4.1]{Barrientos.Gatica.ea:02} that the stress-strain relation for the 
Hencky--Mises model is strongly monotone and Lipschitz-continuous. Also the isotropic damage model satisfies Assumption~\ref{ass:hypo_add} if the damage function in~\eqref{eq:Damage} is, for instance, such that
$$
D(\vec{x},|\matr{\tau}|)=1-(1+\seminorm[d\times d]{\tenf(\vec{x})\matr{\tau}})^{-\frac12}\quad \forall\vec{x}\in\Omega.
$$   
\end{remark} 
In order to prove a convergence rate of $(k+1)$ in space for both the displacement and pressure errors, we assume from this point on that the permeability tensor field $\diff$ is constant on $\Omega$, and that the following elliptic regularity holds (which is the case, e.g., when $\Omega$ is convex \cite{Grisvard:85,Mazya.Rossman:10}):
There is a real number $C_{\rm el}>0$ only depending on $\Omega$ such that, for all $\psi\in L^{2}_{0}(\Omega)$, the unique function $\zeta\in P$  solution of the homogeneous Neumann problem
$$
  -\DIV(\diff\GRAD\zeta)=\psi\quad\text{in $\Omega$},
  \qquad
  \diff\GRAD\zeta\cdot\normal=0\quad\text{on $\partial\Omega$},
$$
is such that 
\begin{equation}\label{eq:ell.reg}
  \norm[H^2(\Omega)]{\zeta}\le C_{\rm el}\ldiff^{-\frac12}\norm[\Omega]{\psi}.
\end{equation}

Let $(\uvu[h]^n, p_h^n)_{1\le n\le N}$ be the solution to~\eqref{eq:nl_biot.h}. We consider, for all $1\le n\le N$, the discrete error components defined as 
\begin{equation}
  \label{eq:err.comp}
  \uve[h]^n\coloneq\uvu[h]^n-\Ivh\ovu^n,\qquad
  \epsilon_h^n\coloneq p_h^n-\tph^n,
\end{equation}
where the global elliptic projection $\tph^n\in\Ph$ is defined as the solution to
$$
c_h(\tph^n,q_h)=c_h(\overline{p}^n,q_h)\quad\forall q_h\in\Ph \quad\text{ and }\quad
\int_\Omega\tph^n=\int_\Omega \overline{p}^n.
$$ 
Before proving the convergence of the scheme, we recall two preliminary approximation results for the projector $\Ivh$ and the projection $\tph^n$ that have been proved in~\cite[Theorem~16]{Botti.Di-Pietro.Sochala:17} and \cite[Lemma 11]{Boffi.Botti.Di-Pietro:16}, respectively.
There is a strictly positive constant $C_{\rm pj}$ depending only on $\Omega$, $k$, and the mesh regularity parameter, such that, 
\begin{itemize}
   \item{Assuming \eqref{eq:hypo_add} and $\vu\in L^2(\vec{U}\cap\Hvec[k+2]{\Th})$ with 
   $\ms(\cdot,\GRADs\vu)\in L^2(\Hmats[k+1]{\Th})$, for a.e. $t\in (0,\tF)$ and all $\uvv[h]\in\UhD$, it holds
  \begin{multline}
    \label{eq:ah.consist}
    \left|a_h(\Ivh\vu(\cdot,t),\uvv[h])+(\DIV\ms(\cdot,\GRADs\vu(\cdot,t)),\vv[h])\right|
    \le 
    \\C_{\rm pj} h^{k+1}\left(
    \seminorm[{\Hvec[k+2]{\Th}}]{\vu(\cdot,t)}+\seminorm[{\Hmat[k+1]{\Th}}]{\ms(\cdot,\GRADs\vu(\cdot,t))}\right)
    \norm[\strain,h]{\uvv[h]}.
  \end{multline}}
  \item{Assuming the elliptic regularity~\eqref{eq:ell.reg} and $\overline{p}^n\in P\cap H^{k+1}(\Th)$, for 
  all $1\le n \le N$, it holds
  \begin{equation}
    \label{eq:tph.approx.L2}
    h\norm[c,h]{\tph^n-\overline{p}^n} + \ldiff^{\frac12}\norm[\Omega]{\tph^n-\overline{p}^n}
    \le C_{\rm pj} h^{k+1} \udiff^{\frac12}\seminorm[H^{k+1}(\Th)]{\overline{p}^n}.
  \end{equation}}
\end{itemize}
Now we have all the ingredients to estimate the discrete errors defined in \eqref{eq:err.comp}.
\begin{theorem}[Error estimate]\label{thm:err_est}
  Let $(\vu,p)$ denote the unique solution to~\eqref{eq:weak_form}, for which we assume
  $$
  \begin{aligned}
    \vu&\in H^1(\mathcal{T}_{\tau};\vec{U})\cap L^2(\Hvec[k+2]{\Th}),\qquad
    \ms(\cdot,\GRADs\vu)\in L^2(\Hmats[k+1]{\Th}),\\
    p&\in L^2(P\cap H^{k+1}(\Th)),
    \qquad\qquad\quad\;\;\phi\in H^1(\mathcal{T}_{\tau};L^2(\Omega)),
  \end{aligned}
  $$
  with $\phi=C_0 p + \DIV\vu$. If $C_0>0$, 
  we further assume $\lproj[\Omega]{0}p \in H^1(\mathcal{T}_{\tau};P^0(\Omega))\eqcolon H^1(\mathcal{T}_{\tau})$.
  Then, under Assumption \ref{ass:hypo_add} and the elliptic regularity~\eqref{eq:ell.reg}, it holds
  \begin{equation}
    \label{eq:err.est}
    \sum_{n=1}^N\tau\norm[\strain,h]{\uve[h]^n}^2
    + \sum_{n=1}^N\tau
    \left(\norm[\Omega]{\epsilon_h^n-\lproj[\Omega]{0}\epsilon_h^n}^2 + C_0\norm[\Omega]{\epsilon_h^n}^2\right)
    + \norm[c,h]{z_h^N}^2
    \le C \left(h^{2k+2}\mathcal{C}_1+\tau^2\mathcal{C}_2\right),
  \end{equation}
  where $C$ is a strictly positive constant independent of $h$, $\tau$, $C_0$, $\diff$, and $\tF$, and, for the sake of 
  brevity, we have defined $z_h^N\coloneq\sum_{n=1}^N\tau\epsilon_h^n$ and introduced the bounded quantities
  $$
  \begin{aligned}
    \mathcal{C}_1&\coloneq \seminorm[L^2({\Hvec[k+2]{\Th}})]{\vu}^2
    +\seminorm[L^2({\Hmat[k+1]{\Th}})]{\ms(\cdot,\GRADs\vu)}^2
    +(1+C_0)\frac{\udiff}{\ldiff}\seminorm[L^2(H^{k+1}(\Th))]{p}^2,
    \\
    \mathcal{C}_2&\coloneq 
    \norm[H^1(\mathcal{T}_{\tau};\Hvec{\Omega})]{\vu}^2+\norm[H^1(\mathcal{T}_{\tau};L^2(\Omega))]{\phi}^2
    +C_0\norm[H^1(\mathcal{T}_{\tau})]{\lproj[\Omega]{0}p}^2.
  \end{aligned}
  $$ 
\end{theorem}
\begin{remark}[Time regularity]
\label{rem:time_reg2}
In order to prove the previous error estimate, we only require the displacement $\vu$ and the fluid content $\phi$ solving problem~\eqref{eq:weak_form} to be piecewise $H^1$-regular in $(0,\tF)$, 
whereas \cite[Theorem 12]{Boffi.Botti.Di-Pietro:16} is established under the much stronger regularity 
$\vu\in C^2(\Hvec{\Omega})$ and, if $C_0>0$, $p\in C^2(L^2(\Omega))$. 
Moreover, the assumptions $\phi\in H^1(\mathcal{T}_{\tau};L^2(\Omega))$ and, 
if $C_0>0$, $\lproj[\Omega]{0}p \in H^1(\mathcal{T}_{\tau})$ are consistent with the time regularity results observed in Remark \ref{rem:reg_porosity.press_average}.
\end{remark}
\begin{proof}
(i) \emph{Estimate of $\norm[\strain,h]{\uve[h]^n}^2$.} First we observe that, owing to \eqref{eq:nl_biot.strong.mech} and the definition of $b_h$ given in \eqref{eq:bh}, for all $\uvv[h]\in\UhD$ and all $1\le n\le N$, we have
\begin{equation}
  \label{eq:rhs.mech}
    (\ovf^n,\vv[h])_\Omega = -(\DIV\oms^n(\cdot,\GRADs\vu),\vv[h])_\Omega + (\GRAD \overline{p}^n,\vv[h])_\Omega
    = a_h(\Ivh\ovu^n,\uvv[h]) + b_h(\uvv[h], \tph^n) - \mathcal{R}^n(\uvv[h]),                      
\end{equation}
where the residual linear form $\mathcal{R}^n:\UhD\to\Real$ is defined such that, for all $\uvv[h]\in\UhD$,
\begin{equation}
  \label{eq:res.fun}
  \mathcal{R}^n(\uvv[h])\coloneq 
  a_h(\Ivh\ovu^n,\uvv[h])+ (\DIV\oms^n(\cdot,\GRADs\vu),\vv[h])_\Omega,
  +b_h(\uvv[h], \tph^n) - (\GRAD \overline{p}^n,\vv[h])_\Omega.
\end{equation}
Using the norm equivalence~\eqref{eq:norm_equivalence}, the strong monotonicity \eqref{eq:hypo_add.smon} of $\ms$ along with assumption \eqref{eq:penalty_mech} on the stabilization parameter, the discrete mechanical equilibrium \eqref{eq:nl_biot.h:mech}, and~\eqref{eq:rhs.mech}, yields 
$$
  \begin{aligned}
    \frac{C_{\rm mn}^{2}}{C_{\rm eq}^{2}}\norm[\strain,h]{\uve[h]^n}^2 
    &=C_{\rm mn}^{2}\sum_{T\in\Th}\left(\norm[T]{\GTs\uve[T]^n}^2 
     +\sum_{F\in\Fh[T]}\frac{1}{h_F}\norm[F]{\vec{\varDelta}_{TF}^k\uve[T]^n}^2\right)
    \\
    &\le a_h(\uvu[h]^n,\uve[h]^n) - a_h(\Ivh\ovu^n,\uve[h]^n) 
    \\
    &= (\ovf^n,\vec{e}_h^n)_\Omega - b_h(\uve[h]^n, p_h^n) - a_h(\Ivh\ovu^n,\uve[h]^n)
    = -b_h(\uve[h]^n, \epsilon_h^n) 
    - \mathcal{R}^n(\uve[h]^n). 
  \end{aligned}
  $$
  Thus, owing to the previous relation and defining the dual norm
    $$\norm[\strain,h,*]{\mathcal{R}^n}\coloneq\sup_{\uvv[h]\in\UhD\setminus\{\underline{\vec{0}}\}}
    \frac{\mathcal{R}^n(\uvv[h]^n)}{\norm[\strain,h]{\uvv[h]^n}},$$ we have that
    $$
    \frac{C_{\rm mn}^2}{C_{\rm eq}^2}\norm[\strain,h]{\uve[h]^n}^2 + b_h(\uve[h]^n, \epsilon_h^n) 
    \le\norm[\strain,h,*]{\mathcal{R}^n}\norm[\strain,h]{\uve[h]^n}
    \le
    \frac{C_{\rm eq}^2}{2 C_{\rm mn}^2} \norm[\strain,h,*]{\mathcal{R}^n}^2
    + \frac{C_{\rm mn}^2}{2 C_{\rm eq}^2}\norm[\strain,h]{\uve[h]^n}^2,
    $$
    where the conclusion follows from Young's inequality. 
    Hence, rearranging, we arrive at
    \begin{equation}
      \label{eq:esterr.displacement}
      \frac{C_{\rm mn}^2}{2 C_{\rm eq}^2}\norm[\strain,h]{\uve[h]^n}^2 + b_h(\uve[h]^n, \epsilon_h^n) 
      \le
      \frac{C_{\rm eq}^2}{2 C_{\rm mn}^2} \norm[\strain,h,*]{\mathcal{R}^n}^2
    \end{equation}
  \\
  (ii) \emph{Estimate of $C_0\norm[\Omega]{\epsilon_h^n}^2$.}
Using \eqref{eq:nl_biot.strong.flow}, the fact that $(\DIV\vec{u}(t),1)_\Omega=0$ to insert $\lproj[\Omega]{0}q_h$, and 
the consistency property~\eqref{eq:ch.consist}, we infer that, for all $q_h\in\Ph$ and all $1\le i\le N$,
\begin{equation}
  \label{eq:rhs.flow}
  \begin{aligned}
  (\overline{g}^i, q_h)_\Omega 
  &= (C_0 \overline{\dt p}^i, q_h)_\Omega + (\DIV(\overline{\dt\vu}^i), q_h-\lproj[\Omega]{0}q_h)_\Omega 
  - (\DIV(\diff\GRAD\overline{p}^i), q_h)_\Omega
  \\
  &=\tau^{-1}\int_{t^{i-1}}^{t^i}\dt\left[C_0(p(t), q_h)_\Omega+(\DIV\vu(t),q_h-\lproj[\Omega]{0}q_h)_\Omega\right]{\rm d}t
  + c_h(\overline{p}^i, q_h)
  \\
  &= \ddt\left[C_0(p^i,q_h)_\Omega + (\DIV\vu^i,q_h-\lproj[\Omega]{0}q_h)_\Omega\right] + c_h(\overline{p}^i, q_h).
  \end{aligned}
\end{equation}
Therefore, using the discrete mass conservation equation \eqref{eq:nl_biot.h:flow}, \eqref{eq:bh(vh,1)=0}, the Fortin  property \eqref{eq:fortin.com} , the definition of the elliptic projection $\tph^n$, and \eqref{eq:rhs.flow}, we obtain
\begin{equation}
  \label{eq:bnd.tph.i}
  \begin{aligned}
  C_0(\ddt \epsilon_h^i, q_h)_\Omega &- b_h(\ddt\uve[h]^i, q_h)+ c_h(\epsilon_h^i, q_h)
  \\
  &= (\overline{g}^i, q_h)_\Omega - C_0 (\ddt\tph^i, q_h)_\Omega 
     + b_h(\ddt(\Ivh\ovu^i), q_h-\lproj[\Omega]{0}q_h) - c_h(\tph^i, q_h)
  \\
  &= (\overline{g}^i, q_h)_\Omega- \ddt\left[C_0(\tph^i, q_h)_\Omega
     -(\DIV\ovu^i, q_h-\lproj[\Omega]{0}q_h)_\Omega \right]-c_h(\overline{p}^i, q_h)
  \\
  &= \ddt \left[C_0(p^i-\tph^i, q_h)_\Omega + (\DIV\vu^i-\DIV\ovu^i, q_h-\lproj[\Omega]{0}q_h)_\Omega \right]
  \\
  &= \ddt \left[ C_0(\overline{p}^i-\tph^i, q_h)_\Omega + C_0(\lproj[\Omega]{0}(p^i-\overline{p}^i), q_h)_\Omega
     + (\phi^i-\overline{\phi}^i, q_h-\lproj[\Omega]{0}q_h)_\Omega \right],
  \end{aligned}
\end{equation}
where, in order to pass to the last line, we have inserted $\pm\overline{p}^i$ into the first term inside brackets in the 
third line, we have defined, according to \eqref{eq:nl_biot.strong:initial}, $\phi^i\coloneq C_0 p^i + \DIV\vu^i$ for all 
$0\le i\le N$, and we have used the definition of the global $L^2$-projector $\lproj[\Omega]{0}$. Moreover, setting
$\tph^0\coloneq0$, it follows from the initial condition \eqref{eq:nl_biot.h:initial}, the boundary condition \eqref{eq:nl_biot.strong:bc.u}, and \eqref{eq:bh(vh,1)=0} that
\begin{equation}
  \label{eq:bnd.tph.initial}
  C_0(\epsilon_h^0, q_h)_\Omega - b_h(\uve[h]^0, q_h) 
  = (\phi^0, q_h)_\Omega
  = (\phi^0, q_h-\lproj[\Omega]{0}q_h)_\Omega+C_0(\lproj[\Omega]{0}p^0, q_h)_\Omega.
\end{equation}
For all $1\le n\le N$, summing \eqref{eq:bnd.tph.i} for $1\le i\le n$ with the choice $q_h=\tau\epsilon_h^n$, using \eqref{eq:bnd.tph.initial}, and proceeding as in the second step of the proof of Proposition \ref{pro:a-priori}, leads to 
\begin{multline}
  \label{eq:bnd.tph.n}
  C_0\norm[\Omega]{\epsilon_h^n}^2 - b_h(\uve[h]^n, \epsilon_h^n) + \frac{1}{2\tau} 
  \left(\norm[c,h]{z_h^n}^2 - \norm[c,h]{z_h^{n-1}}^2 \right)
  \\
  \le C_0(\overline{p}^n-\tph^n, \epsilon_h^n)_\Omega + C_0(\lproj[\Omega]{0}(p^n-\overline{p}^n), \epsilon_h^n)_\Omega
     + (\phi^n-\overline{\phi}^n, \epsilon_h^n-\lproj[\Omega]{0}\epsilon_h^n)_\Omega,
\end{multline}
where $z_h^n\coloneq\sum_{i=1}^n \tau\epsilon_h^i$ if $n\ge 1$ and $z_h^0\coloneq0$. We bound the first term in the right-hand side of \eqref{eq:bnd.tph.n} applying the Cauchy--Schwarz and Young inequalities followed by the approximation result \eqref{eq:tph.approx.L2}, yielding
\begin{equation}
  \label{eq:bnd.tph.t1}
  \begin{aligned}
  C_0(\overline{p}^n-\tph^n, \epsilon_h^n)_\Omega &\le  
  C_{\rm pj}^2 h^{2(k+1)}\frac{\udiff}{\ldiff} C_0 \seminorm[H^{k+1}(\Th)]{\overline{p}^n}^2
  +\frac{C_0}{4}\norm[\Omega]{\epsilon_h^n}^2
  \\
  &\le
  C_{\rm pj}^2\frac{h^{2(k+1)}}{\tau}\left(\frac{\udiff}{\ldiff}\right) C_0\seminorm[L^2((t^{n-1},t^n);H^{k+1}(\Th))]{p}^2
  +\frac{C_0}{4}\norm[\Omega]{\epsilon_h^n}^2,
  \end{aligned}
\end{equation}
where, in order to pass to the second line, we have used the Cauchy--Schwarz inequality and adopted the notation 
$\seminorm[L^2((t^{n-1},t^n);H^{m}(\Th))]{\cdot}\coloneq \norm[L^2((t^{n-1},t^n))]{\seminorm[H^{m}(\Th)]{\cdot}}$, 
for any $m\in\Natural$.
We estimate the second and third terms using the Cauchy--Schwarz and the Young inequalities together with the time approximation result \eqref{eq:time_approx} as follows:
\begin{equation}
  \label{eq:bnd.tph.t2}
  \begin{aligned}
  C_0(\lproj[\Omega]{0}(p^n-\overline{p}^n), \epsilon_h^n)_\Omega
  &\le C_0\tau\norm[H^1((t^{n-1},t^n))]{\lproj[\Omega]{0}p}^2 + \frac{C_0}{4}\norm[\Omega]{\epsilon_h^n}^2,
  \\
  (\phi^n-\overline{\phi}^n, \epsilon_h^n-\lproj[\Omega]{0}\epsilon_h^n)_\Omega
  &\le \eta\tau\norm[H^1((t^{n-1},t^n);L^2(\Omega))]{\phi}^2 
  + \frac1{4\eta}\norm[\Omega]{\epsilon_h^n-\lproj[\Omega]{0}\epsilon_h^n}^2,
  \end{aligned}
\end{equation}
with $\eta$ denoting a positive real number that will be fixed later on in the proof. The relation obtained by plugging \eqref{eq:bnd.tph.t1} and \eqref{eq:bnd.tph.t2} into \eqref{eq:bnd.tph.n} reads
\begin{multline}
  \label{eq:bnd.tph.final}
  \hspace{-2mm}
  \frac{C_0}{2}\norm[\Omega]{\epsilon_h^n}^2 - b_h(\uve[h]^n, \epsilon_h^n) + 
  \frac{1}{2\tau}\left(\norm[c,h]{z_h^n}^2 - \norm[c,h]{z_h^{n-1}}^2 \right)
  - \frac1{4\eta}\norm[\Omega]{\epsilon_h^n-\lproj[\Omega]{0}\epsilon_h^n}^2\le
  \\ \;
  C_{\rm pj}^2\frac{h^{2(k+1)}}{\tau}\left(\frac{\udiff}{\ldiff}\right)C_0\seminorm[L^2((t^{n-1},t^n);H^{k+1}(\Th))]{p}^2
  + C_0\tau\norm[H^1((t^{n-1},t^n))]{\lproj[\Omega]{0}p}^2
  + \eta\tau\norm[H^1((t^{n-1},t^n);L^2(\Omega))]{\phi}^2. 
\end{multline}
\\
(iii) \emph{Estimate of $\norm[\Omega]{\epsilon_h^n-\lproj[\Omega]{0}\epsilon_h^n}^2$.}
We proceed as in the first step of the proof of Proposition \ref{pro:a-priori}.
Using the inf-sup condition~\eqref{eq:inf-sup}, \eqref{eq:bh(vh,1)=0} followed by the definition \eqref{eq:err.comp} of the pressure error, the linearity of $b_h$, the mechanical equilibrium equation~\eqref{eq:nl_biot.h:mech}, and \eqref{eq:rhs.mech}, we get, for all $1\le n\le N$,
\begin{equation}
  \label{eq:bnd.pressure.err1}
  \begin{aligned}
    \norm[\Omega]{\epsilon_h^n-\lproj[\Omega]{0}\epsilon_h^n}
    &\le \beta
    \sup_{\uvv[h]\in\UhD\setminus\{\underline{\vec{0}}\}}
    \frac{b_h(\uvv[h],\epsilon_h^n-\lproj[\Omega]{0}\epsilon_h^n)}{\norm[\strain,h]{\uvv[h]}}
    \\
    &=\beta\sup_{\uvv[h]\in\UhD\setminus\{\underline{\vec{0}}\}}
    \frac{b_h(\uvv[h],p_h^n-\tph^n)}{\norm[\strain,h]{\uvv[h]}
    }
    \\
    &= \beta\sup_{\uvv[h]\in\UhD\setminus\{\underline{\vec{0}}\}}
    \frac{(\ovf^n,\vv[h])_\Omega-a_h(\uvu[h]^n,\uvv[h])-b_h(\uvv[h],\tph^n)}
    {\norm[\strain,h]{\uvv[h]}}
    \\
    &= \beta\sup_{\uvv[h]\in\UhD\setminus\{\underline{\vec{0}}\}}
    \frac{a_h(\Ivh\ovu^{n},\uvv[h]) -a_h(\uvu[h]^n,\uvv[h]) - \mathcal{R}^n(\uvv[h])}
    {\norm[\strain,h]{\uvv[h]}}. 
  \end{aligned}
\end{equation}
Moreover, the Lipschitz continuity of the stress-strain function~\eqref{eq:hypo_add.lip}, the Cauchy--Schwarz inequality, assumption \eqref{eq:penalty_mech} on the stabilization parameter $\gamma$ together with \eqref{eq:const_rel2}, and the second inequality in~\eqref{eq:norm_equivalence}, lead to  
\begin{equation}
  \label{eq:bnd_wLipschitz}
  \begin{aligned}
    a_h&(\Ivh\ovu^{n},\uvv[h]) - a_h(\uvu[h]^n,\uvv[h])
    \\ &=
    \sum_{T\in\Th}\left(\int_T{(\ms(\cdot,\GTs\IvT\ovu^{n}) - \ms(\cdot,\GTs\uvu[T]^n)):\GTs\uvv[T]} 
    +\sum_{F\in\Fh[T]}
    \frac{\gamma}{h_F}\int_F\vec{\varDelta}_{TF}^k(\IvT\ovu^{n}-\uvu[T]^n)\cdot\vec{\varDelta}_{TF}^k\uvv[T]\right)
    \\ &\le C_{\rm lp}C_{\rm eq}^2 
    \norm[\strain,h]{\uve[h]^n}\norm[\strain,h]{\uvv[h]}.
  \end{aligned}
\end{equation}
Therefore, plugging the previous bound into the last line of \eqref{eq:bnd.pressure.err1}, yields 
$$
  \norm[\Omega]{\epsilon_h^n-\lproj[\Omega]{0}\epsilon_h^n}\le
  \beta C_{\rm lp} C_{\rm eq}^2 \norm[\strain,h]{\uve[h]^n}
  +\beta \norm[\strain,h,*]{\mathcal{R}^n}.
$$
Squaring and rearranging the previous relation and recalling that, owing to \eqref{eq:const_rel2}, 
$C_{\rm mn}^2\le C_{\rm lp}$, it is inferred that
\begin{equation}
  \label{eq:bnd.pressure.err2}
  \frac{\norm[\Omega]{\epsilon_h^n-\lproj[\Omega]{0}\epsilon_h^n}^2}{2\beta^2 C_{\rm mn}^{-2} C_{\rm lp}^2 C_{\rm eq}^6}
  \le\frac{C_{\rm mn}^2}{C_{\rm eq}^2} \norm[\strain,h]{\uve[h]^n}^2
  +\frac{\norm[\strain,h,*]{\mathcal{R}^n}^2}{C_{\rm mn}^2 C_{\rm eq}^6}.
\end{equation}
\\
(iv) \emph{Estimate of the dual norm of the residual.}
We split the residual linear form $\mathcal{R}^n$ defined in \eqref{eq:res.fun} into three contributions 
$\mathcal{R}^n \coloneq \mathcal{R}_{1}^n+\mathcal{R}_{2}^n+\mathcal{R}_{3}^n$, defined, for all $\uvv[h]\in\UhD$ and all $1\le n\le N$, such that
\begin{subequations}
  \label{eq:res.funs}
  \begin{alignat}{3} 
    \label{eq:res.fun.R1}
    \mathcal{R}_{1}^n(\uvv[h]) &\coloneq 
    a_h(\Ivh\ovu^n,\uvv[h]) - \frac1\tau \int_{t^{n-1}}^{t^n} a_h(\Ivh\vu(t),\uvv[h])\ \ud t,
    \\ 
    \label{eq:res.fun.R2}
    \mathcal{R}_{2}^n(\uvv[h]) &\coloneq 
    (\DIV\oms^n(\cdot,\GRADs\vu),\vv[h])_\Omega+\frac1\tau \int_{t^{n-1}}^{t^n} a_h(\Ivh\vu(t),\uvv[h])\ \ud t,
    \\ 
    \label{eq:res.fun.R3}
    \mathcal{R}_{3}^n(\uvv[h]) &\coloneq 
    b_h(\uvv[h], \tph^n) - (\GRAD \overline{p}^n,\vv[h])_\Omega.
  \end{alignat}
\end{subequations}
The first contribution can be bounded proceeding as in~\eqref{eq:bnd_wLipschitz}, then using the stability property \eqref{eq:fortin.st} of the interpolator $\Ivh$, the Cauchy--Schwarz inequality, and a Poincar\'{e}--Wirtinger inequality on the time interval $(t^{n-1},t^n)$. By doing so, we get
\begin{equation}\label{eq:bnd_Res1}
\begin{aligned}
\mathcal{R}_{1}^n(\uvv[h]) &= \frac1\tau \int_{t^{n-1}}^{t^n} a_h(\Ivh\ovu^n,\uvv[h]) - a_h(\Ivh\vu(t),\uvv[h])\ \ud t
\\
&\le \frac1\tau \int_{t^{n-1}}^{t^n} 
\left( C_{\rm lp}C_{\rm eq}^2 \norm[\strain,h]{\Ivh(\ovu^n -\vu(t))}\norm[\strain,h]{\uvv[h]} \right) \ud t
\\
&\le \frac{C_{\rm lp}C_{\rm eq}^2 C_{\rm st}}{\tau} \norm[\strain,h]{\uvv[h]} 
\int_{t^{n-1}}^{t^n}\norm[1,\Omega]{\ovu^n -\vu(t)} \ud t
\\
&\le \frac{C_{\rm lp}C_{\rm eq}^2 C_{\rm st}}{\sqrt{\tau}} \norm[\strain,h]{\uvv[h]} 
\norm[L^2((t^{n-1},t^n);\Hvec{\Omega})]{\ovu^n -\vu}
\\
&\le C_{\rm lp}C_{\rm eq}^2 C_{\rm st} C_{\rm ap} \sqrt{\tau} \norm[\strain,h]{\uvv[h]} 
\norm[H^1((t^{n-1},t^n);\Hvec{\Omega})]{\vu}.
\end{aligned}
\end{equation}
We estimate the residual linear form $\mathcal{R}_2^n$ defined in \eqref{eq:res.fun.R2} using the consistency property \eqref{eq:ah.consist} and the Cauchy--Schwarz inequality, obtaining
\begin{equation}\label{eq:bnd_Res2}
\begin{aligned}
\mathcal{R}_{2}^n(\uvv[h]) &= 
\frac1\tau \int_{t^{n-1}}^{t^n} \left(\DIV\ms(\cdot,\GRADs\vu(t)),\uvv[h]\right)_\Omega + a_h(\Ivh\vu(t),\uvv[h])\ \ud t
\\
&\le C_{\rm pj}\frac{h^{k+1}}\tau \norm[\strain,h]{\uvv[h]} \int_{t^{n-1}}^{t^n} 
\left( \seminorm[{\Hvec[k+2]{\Th}}]{\vu(t)}+\seminorm[{\Hmat[k+1]{\Th}}]{\ms(\cdot,\GRADs\vu(t))}\right) \ud t
\\
&\le C_{\rm pj}\frac{h^{k+1}}{\sqrt{\tau}} \norm[\strain,h]{\uvv[h]} 
\left( \seminorm[L^2((t^{n-1},t^n);{\Hvec[k+2]{\Th}})]{\vu}
+\seminorm[L^2((t^{n-1},t^n);{\Hmat[k+1]{\Th}})]{\ms(\cdot,\GRADs\vu)}\right).
\end{aligned}
\end{equation}
Finally, the third term in \eqref{eq:res.funs} can be bounded integrating by parts element-wise and using the Cauchy--Schwarz inequality, the trace inequality \eqref{eq:continuous_trace}, the consistency property \eqref{eq:tph.approx.L2}, and again the Cauchy--Schwarz inequality, namely
\begin{equation}\label{eq:bnd_Res3}
\begin{aligned}
\mathcal{R}_{3}^n(\uvv[h]) 
     &\le \sum_{T\in\Th}\int_T \DIV\vv[T](\overline{p}^n - \tph^n) 
     + \sum_{F\in\Fh[T]}\int_F (\vv[F]-\vv[T])\cdot (\overline{p}^n-\tph^n)\normal_{TF}
     \\
     &\le C_{\rm pj} h^{k+1}
    \left(\frac{\udiff}{\ldiff}\right)^{\frac12}
     \seminorm[H^{k+1}(\Th)]{\overline{p}^n}\norm[\strain,h]{\uvv[h]}
     \\
     &\le C_{\rm pj}\frac{h^{k+1}}{\sqrt{\tau}}  \norm[\strain,h]{\uvv[h]}
     \left(\frac{\udiff}{\ldiff}\right)^{\frac12} \seminorm[L^2((t^{n-1},t^n);H^{k+1}(\Th))]{p}.
\end{aligned}
\end{equation}
Therefore, combining \eqref{eq:bnd_Res1}, \eqref{eq:bnd_Res2}, and \eqref{eq:bnd_Res3}, it is inferred that
\begin{equation}
  \label{eq:bnd_Res.final}
\norm[\strain,h,*]{\mathcal{R}^n}^2
= \norm[\strain,h,*]{\mathcal{R}_1^n + \mathcal{R}_2^n + \mathcal{R}_3^n}^2
\le 4 C_{\rm lp}^2 C_{\rm eq}^4 C_{\rm st}^2 C_{\rm ap}^2 \tau \norm[H^1((t^{n-1},t^n);\Hvec{\Omega})]{\vu}^2
+4 C_{\rm pj}^2 \tau^{-1} h^{2(k+1)} \widetilde{\mathcal{C}}_1^n,
\end{equation}
with
$$
\widetilde{\mathcal{C}}_1^n\coloneq
\seminorm[L^2((t^{n-1},t^n);{\Hvec[k+2]{\Th}})]{\vu}^2
+\seminorm[L^2((t^{n-1},t^n);{\Hmat[k+1]{\Th}})]{\ms(\cdot,\GRADs\vu)}^2
+\left(\frac{\udiff}{\ldiff}\right) \seminorm[L^2((t^{n-1},t^n);H^{k+1}(\Th))]{p}^2.
$$
\\
(v) \emph{Conclusion.}
Adding~\eqref{eq:esterr.displacement} to~\eqref{eq:bnd.tph.final} with $\eta=4\beta^2 C_{\rm mn}^{-2} C_{\rm lp}^2 
C_{\rm eq}^6$, using \eqref{eq:bnd.pressure.err2} and \eqref{eq:bnd_Res.final}, summing the resulting equation over $1\le n \le N$, and multiplying both sides by $2\tau$, we obtain
\begin{equation}
  \label{eq:err.eq.sum}
    \sum_{n=1}^N\tau \left(\frac{C_{\rm mn}^2}{2 C_{\rm eq}^2}\norm[\strain,h]{\uve[h]^n}^2
    + C_0\norm[\Omega]{\epsilon_h^n}^2 
    +\frac{\norm[\Omega]{\epsilon_h^n-\lproj[\Omega]{0}\epsilon_h^n}^2}{8\beta^2C_{\rm mn}^{-2}C_{\rm lp}^2C_{\rm eq}^6}
    \right) + \norm[c,h]{z_h^N}^2
    \le \widetilde{C} \left(h^{2(k+1)}\mathcal{C}_1+\tau^2 \mathcal{C}_2\right),
\end{equation}
with $\widetilde{C}\coloneq\max\left\{1,\, C_{\rm pj}^2,
\, 2 C_{\rm pj}^2 C_{\rm mn}^{-2}(2 C_{\rm eq}^2 + C_{\rm eq}^{-6}),
\, 2 C_{\rm lp}^2 C_{\rm st}^2 C_{\rm ap}^2 (2 C_{\rm eq}^6 + C_{\rm eq}^{-2}),
\,4\beta^2 C_{\rm lp}^2 C_{\rm eq}^6 C_{\rm mn}^{-2}\right\}$ and
$$
  \begin{aligned}
    \mathcal{C}_1 &\coloneq \sum_{n=1}^N
    \left(\widetilde{\mathcal{C}}_1^n+C_0\frac{\udiff}{\ldiff}\seminorm[L^2((t^{n-1},t^n);H^{k+1}(\Th))]{p}^2\right),
    \\
    &=\seminorm[L^2({\Hvec[k+2]{\Th}})]{\vu}^2+\seminorm[L^2({\Hmat[k+1]{\Th}})]{\ms(\cdot,\GRADs\vu)}^2
    +(1+C_0)\frac{\udiff}{\ldiff}\seminorm[L^2(H^{k+1}(\Th))]{p}^2
    \\
    \mathcal{C}_2 &\coloneq \sum_{n=1}^N\left(
    \norm[H^1((t^{n-1},t^n);\Hvec{\Omega})]{\vu}^2
    +\norm[H^1((t^{n-1},t^n);L^2(\Omega))]{\phi}^2
    +C_0\norm[H^1((t^{n-1},t^n))]{\lproj[\Omega]{0}p}^2\right)
    \\
    &= \norm[H^1(\mathcal{T}_{\tau};\Hvec{\Omega})]{\vu}^2+\norm[H^1(\mathcal{T}_{\tau};L^2(\Omega))]{\phi}^2
    +C_0\norm[H^1(\mathcal{T}_{\tau})]{\lproj[\Omega]{0}p}^2.
  \end{aligned}
$$ 
Finally, multiplying both sides of \eqref{eq:err.eq.sum} by $2(C_{\rm pj}^{-2}+1)\widetilde{C}$ yields  \eqref{eq:err.est} with $C=2(C_{\rm pj}^{-2}+1)\widetilde{C}^2$.
\end{proof}


\section{Numerical results}\label{sec:num.res}

\begin{figure}
  \centering
    \includegraphics[height=4cm]{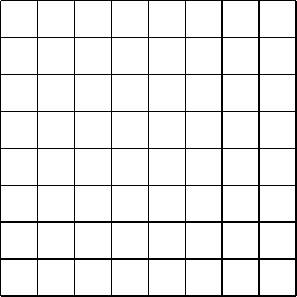}
    \hspace{1cm}
    \includegraphics[height=4cm]{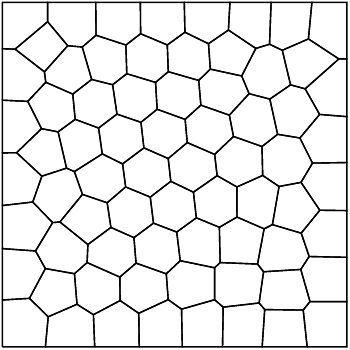}
  \caption{Cartesian and Voronoi meshes for the numerical tests.\label{fig:meshes}}
\end{figure}

We consider a regular exact solution in order to assess the convergence of the method. 
Specifically, we solve problem~\eqref{eq:nl_biot.strong} in the square domain $\Omega=(0,1)^{2}$ with $\tF=1$ and physical 
parameters $C_0=0$ and $\diff=\Id$. As nonlinear constitutive law we take the Hencky--Mises relation given by
$$
  \ms(\GRADs\vu)=\left(1+\exp^{-\opdev(\GRADs\vu)}\right)\optr(\GRADs\vu)\Id+\left(4-2\exp^{-\opdev(\GRADs\vu)}\right)\GRADs\vu.
$$
It can be checked that the previous stress-strain relation satisfies Assumption~\ref{ass:hypo}. The exact displacement $\vu$ and exact pressure $p$ are given by
$$
\begin{aligned}
  \vu(\vec{x},t) &= t^2\big(\sin(\pi x_1)\sin(\pi x_2), \;\sin(\pi x_1)\sin(\pi x_2)\big),
  \\
  p(\vec{x},t) &= -\pi^{-1}t(\sin(\pi x_1)\cos(\pi x_2) + \cos(\pi x_1)\sin(\pi x_2)).
\end{aligned}
$$
The volumetric load $\vf$, the source term $g$, and the boundary conditions are inferred from the exact solution.
We consider the Cartesian and Voronoi mesh families depicted in Figure~\ref{fig:meshes} and polynomial degree $k=1$.
The time step $\tau$ on the coarsest mesh is taken to be $\pgfmathprintnumber{0.2}/2^{k+1}$ for every choice of the polynomial degree $k$, and it decreases with the mesh size $h$ according to the theoretical convergence rates, thus, 
$\tau_{l} /\tau_{l+1} = 2^k h_{l}/h_{l+1}$.
Table~\ref{tab:convergence} displays convergence results for the two mesh families. 
The error measures are $\left(\sum_{n=1}^N\tau\norm[\strain,h]{\uvu[h]^{n}-\Ih\ovu^{n}}^2\right)^{\frac12}$ for the displacement and $\left(\sum_{n=1}^N\tau\norm[\Omega]{p_h^n-\lproj[h]{k}\overline{p}^n}^2\right)^{\frac12}$ for the pressure.
In all cases, the orders of convergence are in agreement with the theoretical predictions. In particular, it is observed that the optimal convergence rates stated in Theorem~\ref{thm:err_est} are reached.

\begin{table}\centering
  \caption{Convergence results on the Cartesian and Voronoi meshes for $k=1$. 
    OCV stands for order of convergence.
  \label{tab:convergence}}
  \begin{tabular}{ccccc}
    \toprule
    $h$ & $\left(\sum_{n=1}^N\tau\norm[\strain,h]{\uvu[h]^{n}-\Ih\ovu^{n}}^2\right)^{\frac12}$ & OCV 
    & $\left(\sum_{n=1}^N\tau\norm[\Omega]{p_h^n-\lproj[h]{k}\overline{p}^n}^2\right)^{\frac12}$ & OCV \\
    \midrule\multicolumn{5}{c}{Cartesian mesh family} \\ \midrule    
    \pgfmathprintnumber[zerofill]{6.25e-02}
    &  \pgfmathprintnumber[zerofill]{3.10e-02}
    &  ---
    &  \pgfmathprintnumber[zerofill]{3.87e-01}
    &  ---                               \\
    \pgfmathprintnumber[zerofill]{3.12e-02}
    &  \pgfmathprintnumber[zerofill]{8.52e-03}
    &  \pgfmathprintnumber[fixed zerofill,precision=2]{1.86}
    &  \pgfmathprintnumber[zerofill]{9.65e-02}
    &  \pgfmathprintnumber[fixed zerofill,precision=2]{2.00}         \\
    \pgfmathprintnumber[zerofill]{1.56e-02}
    &  \pgfmathprintnumber[zerofill]{2.22e-03}
    &  \pgfmathprintnumber[fixed zerofill,precision=2]{1.94}
    &  \pgfmathprintnumber[zerofill]{2.44e-02}
    &  \pgfmathprintnumber[fixed zerofill,precision=2]{1.98}         \\
    \pgfmathprintnumber[zerofill]{7.81e-03}
    &  \pgfmathprintnumber[zerofill]{5.61e-04}
    &  \pgfmathprintnumber[fixed zerofill,precision=2]{1.99}
    &  \pgfmathprintnumber[zerofill]{6.18e-03}
    &  \pgfmathprintnumber[fixed zerofill,precision=2]{1.99}         \\
    \pgfmathprintnumber[zerofill]{3.91e-03}
    &  \pgfmathprintnumber[zerofill]{1.41e-04}
    &  \pgfmathprintnumber[fixed zerofill,precision=2]{2.00}
    &  \pgfmathprintnumber[zerofill]{1.56e-03}
    &  \pgfmathprintnumber[fixed zerofill,precision=2]{1.99}         \\
    \midrule\multicolumn{5}{c}{Voronoi mesh family} \\ \midrule
    \pgfmathprintnumber[zerofill]{6.50e-02}
    &  \pgfmathprintnumber[zerofill]{3.28e-02}
    &  ---
    &  \pgfmathprintnumber[zerofill]{2.73e-01}
    &  ---                               \\
    \pgfmathprintnumber[zerofill]{3.15e-02}
    &  \pgfmathprintnumber[zerofill]{8.48e-03}
    &  \pgfmathprintnumber[fixed zerofill,precision=2]{1.87}
    &  \pgfmathprintnumber[zerofill]{6.58e-02}
    &  \pgfmathprintnumber[fixed zerofill,precision=2]{1.96}         \\
    \pgfmathprintnumber[zerofill]{1.61e-02}
    &  \pgfmathprintnumber[zerofill]{2.20e-03}
    &  \pgfmathprintnumber[fixed zerofill,precision=2]{2.01}
    &  \pgfmathprintnumber[zerofill]{1.63e-02}
    &  \pgfmathprintnumber[fixed zerofill,precision=2]{2.08}         \\
    \pgfmathprintnumber[zerofill]{9.09e-03}
    &  \pgfmathprintnumber[zerofill]{5.72e-04}
    &  \pgfmathprintnumber[fixed zerofill,precision=2]{2.36}
    &  \pgfmathprintnumber[zerofill]{4.24e-03}
    &  \pgfmathprintnumber[fixed zerofill,precision=2]{2.36}         \\
    \pgfmathprintnumber[zerofill]{4.26e-03}
    &  \pgfmathprintnumber[zerofill]{1.42e-04}
    &  \pgfmathprintnumber[fixed zerofill,precision=2]{1.83}
    &  \pgfmathprintnumber[zerofill]{1.05e-03}
    &  \pgfmathprintnumber[fixed zerofill,precision=2]{1.84}         \\
    \bottomrule
  \end{tabular}
\end{table}


\section*{Acknowledgements}

This work was partially funded by the Bureau de Recherches G\'{e}ologiques et Mini\`{e}res. The work of M. Botti was additionally partially supported by Labex NUMEV (ANR-10-LABX-20) ref. 2014-2-006.
The work of D. A. Di Pietro was additionally partially supported by project HHOMM (ANR-15-CE40-0005).

\bibliographystyle{plain}
\bibliography{nolibho}

\end{document}